\documentclass[12pt,leqno]{amsart}

\usepackage{enumitem}
\setenumerate[1]{label=(\alph*)}
\setenumerate[2]{label=(\roman*)}

\usepackage[english]{babel}

\usepackage{amssymb}

\usepackage[all]{xy}
\usepackage{ulem}
\usepackage{upgreek}

\usepackage{xr-hyper}
\usepackage{hyperref}
\hypersetup{colorlinks=true, anchorcolor=red, citecolor= green,
filecolor=magenta, linkcolor=red, menucolor=red, urlcolor=cyan}

\newtheorem{theorem}{Theorem}[section]
\newtheorem{lemma}[theorem]{Lemma}
\newtheorem{definition}[theorem]{Definition}
\newtheorem{proposition}[theorem]{Proposition}

\newtheorem{corollary}[theorem]{Corollary}
\newtheorem{remark}[theorem]{Remark}

\newtheorem{example}[theorem]{Example}

\numberwithin{equation}{section}

\newcommand{\colim}{\operatorname{colim}}
\newcommand{\add}{\operatorname{add}}
\newcommand{\Aut}{\operatorname{Aut}}

\newcommand{\D}{{\operatorname{D}}}
\newcommand{\bd}{{\op{b}}}
\newcommand{\bR}{{\mathbf{R}}}

\newcommand{\bfMF}{{\mathbf{MF}}}
\newcommand{\HP}{{\operatorname{HP}}}
\newcommand{\opH}{{\operatorname{H}}}
\newcommand{\opc}{{\op{c}}}
\newcommand{\opSH}{{\operatorname{SH}}}
\newcommand{\an}{{\op{an}}}
\newcommand{\geom}{{\op{geom}}}

\newcommand{\Spec}{\operatorname{Spec}}

\newcommand{\id}{\operatorname{id}}

\newcommand{\ra}{\rightarrow}

\newcommand{\hra}{\hookrightarrow}

\newcommand{\xra}[1]{\xrightarrow{#1}}

\newcommand{\sira}{\xra{\sim}}

\newcommand{\op}[1]{{\operatorname{#1}}}

\newcommand{\DZ}{\mathbb{Z}}
\newcommand{\DN}{\mathbb{N}}

\newcommand{\DC}{\mathbb{C}}

\newcommand{\DA}{{\mathbb{A}}}
\newcommand{\DP}{{\mathbb{P}}}

\newcommand{\DL}{\mathbb{L}}

\newcommand{\DG}{{\mathbb{G}}}

\newcommand{\DGm}{{\DG_\op{m}}}

\newcommand{\Irr}{\op{Irr}}

\newcommand{\Coh}{{\op{Coh}}}

\newcommand{\ol}[1]{{\overline{#1}}}
\newcommand{\ul}[1]{{\underline{#1}}}

\newcommand{\sat}{\op{sat}}

\newcommand{\tildew}[1]{\widetilde{#1}}

\newcommand{\Var}{{\op{Var}}}

\newcommand{\define}[1]{{\textbf{#1}}}

\renewcommand{\epsilon}{{\varepsilon}}

\newcommand{\Sing}{{\op{Sing}}}
\newcommand{\sing}{{\op{sing}}}
\newcommand{\reg}{{\op{reg}}}
\newcommand{\Crit}{{\op{Crit}}}
\newcommand{\kk}{{\mathsf{k}}}

\author{Valery A.~Lunts \and Olaf M.~Schn{\"u}rer}

\address{
  Department of Mathematics\\
  Indiana University\\
  Rawles Hall\\
  831 East 3rd Street\\
  Bloomington, IN 47405\\
  USA
}

\email{vlunts@indiana.edu} 

\address{
  Mathematisches Institut\\ 
  Universit{\"a}t Bonn\\
  Endenicher Allee 60\\
  53115 Bonn\\
  Germany
}

\email{olaf.schnuerer@math.uni-bonn.de}

\title{Motivic vanishing cycles as a motivic measure}

\thanks{}

\begin{document}

\begin{abstract}
  We show that the motivic vanishing cycles introduced by J.\
  Denef and F.\ Loeser give rise to a motivic
  measure on the Grothendieck ring of varieties over the affine
  line. We discuss the relation of this motivic measure to the 
  motivic measure we constructed earlier using categories of
  matrix factorizations.
\end{abstract}

\maketitle
\setcounter{tocdepth}{1}

\tableofcontents

\section{Introduction}
\label{sec:introduction}

The motivic nearby fiber and the motivic vanishing cycles were
introduced by J.~Denef and F.~Loeser (see
\cite{denef-loeaser-motivic-igusa-zeta, denef-loeser-motivic,
  denef-loeser-arc-spaces, looijenga-motivic-measures}). Let $V
\colon X \ra \DA^1_\kk$ 
be a morphism of $\kk$-varieties where $\kk$ is an algebraically
closed field of characteristic zero and $X$ is smooth over $\kk$
and connected.
The motivic nearby fiber $\psi_{V,a}$ and the motivic vanishing
cycles $\phi_{V,a}$ of $V$
at a point $a \in \kk=\DA^1_\kk(\kk)$ are elements of
a localization $\mathcal{M}_{|X_a|}^{\hat{\upmu}}$ of the
equivariant Grothendieck ring $K_0(\Var_{|X_a|}^{\hat{\upmu}})$ 
of varieties over the reduced fiber $|X_a|$ of $V$ over $a$.
We refer the reader to the main body of this article for precise
definitions. We will often view $\psi_{V,a}$ and $\phi_{V,a}$ as
elements of $\mathcal{M}_{\kk}^{\hat{\upmu}}$ in this introduction.

The motivic nearby fiber is additive on the Grothendieck group
$K_0(\Var_{\DA^1_\kk})$ of varieties over $\DA^1_\kk$, as shown
by F.~Bittner 
\cite{bittner-motivic-nearby} 
and by G.~Guibert, F.~Loeser and M.~Merle 
\cite[Thm.~3.9]{guibert-loeser-merle-convolution}.
Namely, for any $a \in \kk$,
there is a map
\begin{equation*}
  K_0(\Var_{\DA^1_\kk}) \ra \mathcal{M}_\kk^{\hat{\upmu}}  
\end{equation*}
of $K_0(\Var_\kk)$-modules which maps the class of a proper
morphism $V \colon X \ra \DA^1_\kk$ with $X$ as above to 
the motivic nearby fiber $\psi_{V,a}$.
  
The motivic Thom-Sebastiani theorem
\cite{guibert-loeser-merle-convolution} 
is a
local multiplicativity result for motivic vanishing cycles. Given
another morphism $W \colon Y \ra \DA^1_\kk$ as above define
$V \circledast W \colon X \times Y \ra \DA^1_\kk$ by $(V
\circledast W)(x,y)=V(x)+W(y)$. Then the motivic
Thom-Sebastiani Theorem states that a certain convolution of the
motivic vanishing cycles $\phi_{V,a}$ and $\phi_{W,b}$ determines
some part of the motivic vanishing cycles
$\phi_{V \circledast W,a+b}$ (see Theorem
\ref{t:thom-sebastiani-GLM-motivic-vanishing-cycles}).

Our main result states that after small adjustments - the
motivic vanishing cycles $\phi_{V,a}$ we use differ by a factor
$(-1)^{\dim X}$ from the usual motivic vanishing cycles (see
Remark~\ref{rem:motivic-fibers-signs}) - the motivic vanishing
cycles are 
both additive and multiplicative.

\begin{theorem}
  [{see Theorem~\ref{t:total-motivic-vanishing-cycles-ring-homo}}]
  \label{t:main-intro}
  There is a morphism
  \begin{equation}
    \label{eq:60}
    (K_0(\Var_{\DA^1_\kk}), \star) \ra
    (\mathcal{M}_\kk^{\hat{\upmu}},*) 
  \end{equation}
  of $K_0(\Var_\kk)$-algebras - called \textbf{motivic vanishing
    cycles measure} -
  which is uniquely
  determined by the following 
  property: it maps the class of each proper
  morphism $V \colon  X \ra \DA^1_\kk$ from a smooth and connected
  $\kk$-variety $X$ to the sum 
  $\sum_{a \in \kk} \phi_{V,a}$ of its motivic vanishing cycles.
\end{theorem}

The motivic vanishing cycles measure is a motivic
measure in the sense that it is a ring morphism from
some Grothendieck ring of
varieties to another ring.
The multiplication $*$ on the target of our measure is a convolution
product whose definition is due to Looijenga and involves Fermat
varieties. The 
multiplication $\star$ on the source is given by
$[X \xra{V} \DA^1_\kk] \star [Y \xra{W} \DA^1_\kk] = 
[X \times Y \xra{V \circledast W} \DA^1_\kk]$.
Apart from the additivity and local multiplicativity results
mentioned 
above, the main ingredient in the proof of
Theorem~\ref{t:main-intro} is a compactification construction
described in 
\cite{valery-olaf-matfak-motmeas}. 
In fact, 
we prove a slightly stronger statement
in Theorem~\ref{t:total-motivic-vanishing-cycles-ring-homo}: the
motivic 
vanishing cycles 
measure \eqref{eq:60}
comes from a
morphism $(K_0(\Var_{\DA^1_\kk}), \star) \ra
(\tilde{\mathcal{M}}_{\DA^1_\kk}^{\hat{\upmu}},\star)$ 
of $K_0(\Var_\kk)$-algebras. 
Let us mention that our sign adjustments are already
necessary for additivity (see
Remark~\ref{rem:justify-sign-vanishing}).

In the last part of this article we compare the motivic vanishing
cycles measure 
with a motivic measure of a completely different categorical
nature (in case 
$\kk=\DC$). Mapping a projective morphism $W \colon X \ra \DA^1_\DC$
from a smooth complex variety $X$ to its category of matrix
factorizations gives rise to
a ``matrix factorization'' motivic
measure
\begin{equation*}
  \mu \colon 
  (K_0(\Var_{\DA^1_\DC}), \star)  
  \ra K_0(\sat_\DC^{{\DZ}_2})  
\end{equation*}
as we explained in 
\cite{valery-olaf-matfak-semi-orth-decomp,
  valery-olaf-matfak-motmeas}. The target of this ring morphism
is the Grothendieck
ring of saturated differential $\DZ_2$-graded categories.
Here is our comparison result.

\begin{theorem} 
  [{see Theorem~\ref{t:comparison}}]
  \label{t:comparison-intro}
  We have the following commutative diagram of ring homomorphisms
  \begin{equation*}
    \xymatrix{
      {(K_0(\Var_{\DA^1_\DC}), \star)} 
      \ar[r]^-{\mu} \ar[d]
      & {K_0(\sat_{\DC}^{{\DZ}_2})} \ar[d]^-{\chi_{\HP}} \\
      {(\mathcal{M}_{\DC}^{\hat{\upmu}},*)} \ar[r]^-{\chi_\opc} 
      & {\DZ}
    }
  \end{equation*}
  where the left vertical arrow is 
  the motivic vanishing cycles measure \eqref{eq:60}
  from Theorem~\ref{t:main-intro},
  the lower horizontal arrow is induced by forgetting the group
  action and taking the Euler
  characteristic with compact support, and the 
  right vertical arrow is induced by taking the Euler
  characteristic of periodic cyclic homology.
\end{theorem}

The main ingredients in the proof of this theorem are the
comparison between the periodic cyclic homology of the dg
category of matrix factorizations of a given potential $V$ with
the vanishing cohomology of $V$ due to A.~Efimov
\cite{efimov-cyclic-homology-matrix-factorizations}, and the
comparison between the motivic and geometric vanishing cycles due
to G.~Guibert, F.~Loeser and M.~Merle
\cite{guibert-loeser-merle-convolution}.
 
\subsection{Structure of the article}
\label{sec:structure-article}

\begin{itemize}
\item[\S\ref{sec:grothendieck-rings}]
  We remind the reader of various (equivariant) Grothendieck
  abelian groups of 
  varieties and multiplications (or ``convolutions'') on them.
  We recall Looijenga's convolution
  product $*$ in section~\ref{sec:convolution} and include a direct
  proof of associativity (see
  Proposition~\ref{p:convolution-comm-ass-unit}); 
  this reproves results of 
  \cite[5.1-5.5]{guibert-loeser-merle-convolution}.
  We also define a variant of Looijenga's convolution product for
  varieties over $\DA^1_\kk$ in section~\ref{sec:conv-vari-over}.
\item[\S\ref{sec:motiv-vanish-fiber}]
  We recall the definition of the motivic nearby
  fiber $\psi_{V,a}$ and the motivic vanishing
  cycles $\phi_{V,a}$ and show that $\phi_{V,a}$ lies in 
  $\mathcal{M}^{\hat{\upmu}}_{|\Sing(V) \cap X_a|}$ (see
  Proposition~\ref{p:motivic-vanishing-cycles-over-singular-part}). 
  We also show an invariance property of $\phi_{V,a}$ 
  in Corollary~\ref{c:mot-van-fiber-compactification}.
\item[\S\ref{sec:thom-sebastiani}]
  We state the motivic Thom-Sebastiani Theorem
  \cite[Thm.~5.18]{guibert-loeser-merle-convolution}
  as
  Theorem~\ref{t:thom-sebastiani-GLM-motivic-vanishing-cycles} and
  give some corollaries. In particular, we globalize the
  Thom-Sebastiani Theorem to 
  Corollary~\ref{c:global-thom-sebastiani}.
\item[\S\ref{sec:motiv-vanish-fiber-1}]
  A corollary of
  \cite[Thm.~3.9]{guibert-loeser-merle-convolution} is given as
  Theorem~\ref{t:GLM-measure}. We obtain additivity of the
  motivic vanishing cycles in 
  Theorem~\ref{t:total-vanishing-cycles-group-homo}.
  Then we deduce our main
  Theorem~\ref{t:total-motivic-vanishing-cycles-ring-homo} 
  using the previous Thom-Sebastiani results
  and
  the compactification result stated as
  Proposition~\ref{p:compactification}.
\item [\S\ref{sec:comp-with-matr}]
  We remind the reader of the categorical motivic measure
  in \cite{bondal-larsen-lunts-grothendieck-ring}
  and its relation to the matrix factorization measure.
  Then we prove Theorem~\ref{t:comparison}. We finish by giving
  two examples and by drawing a diagram relating the motivic
  measures considered in this article.
\end{itemize}

\subsection{Acknowledgements}
\label{sec:acknowledgements}

We thank Daniel Bergh, Alexander Efimov, Annabelle Hartmann,
Fran\-\c{c}ois Loeser, and 
Anatoly Preygel for useful discussions, and the referee for
careful reading.

The first author was supported by 
NSA grant 43.294.03.
The second author was supported by a postdoctoral
fellowship of the DFG
and by 
SFB/TR 45 of the DFG.

\subsection{Conventions}
\label{sec:conventions}

We fix an algebraically closed field $\kk$ of characteristic
zero.
By a $\kk$-variety we mean a separated reduced scheme of finite
type over $\kk$.
A morphism of $\kk$-varieties is a morphism of $\kk$-schemes.
Let $\Var_\kk$ be the category of
$\kk$-varieties.
We write $\times$ instead of $\times_{\Spec \kk}$. By our
assumptions on $\kk$, the product of two $\kk$-varieties is again
reduced and hence a $\kk$-variety.
If $X$ is a scheme we denote by $|X|$ the
corresponding reduced closed subscheme. 

\section{Grothendieck rings of varieties}
\label{sec:grothendieck-rings}

\subsection{Grothendieck rings of varieties over a base variety}
\label{sec:groth-rings-vari}


Fix a $\kk$-variety $S$. By an $S$-variety we mean a morphism
$X \ra S$ of $\kk$-varieties. Let $\Var_S$ be the category of
$S$-varieties.
The Grothendieck group $K_0(\Var_S)$ of $S$-varieties is the
quotient of the free abelian group on 
isomorphism 
classes $\langle X \ra S\rangle$ of $S$-varieties $X \ra S$ by the subgroup generated by the scissor
expressions $\langle X \ra S \rangle - \langle (X - Y) \ra S
\rangle - \langle Y \ra S\rangle$ where $Y \subset X$ is a
closed reduced subscheme.
Any $S$-variety $X \ra S$  
defines an element $[X \ra S]$ of $K_0(\Var_S)$. 

Given $S$-varieties 
$X \ra S$ and $Y \ra S$, the composition
$|X \times_S Y| \ra X \times_S Y \ra S$ is an $S$-variety; this
operation turns $K_0(\Var_S)$ into a commutative associative ring
with identity element
$[S \xra{\id} S]$ (use \cite[Prop.~4.34]{goertz-wedhorn-AGI} for
associativity). 

Let $\mathcal{M}_S:= K_0(\Var_S)[\DL_S^{-1}]$ be the ring obtained
from  
$K_0(\Var_S)$ by inverting $\DL_S=[\DA^1_S \ra S]$.

We usually write $K_0(\Var_\kk)$ instead of $K_0(\Var_{\Spec
  \kk})$, $\DL=\DL_\kk$ instead of $\DL_{\Spec \kk}$, and $\mathcal{M}_\kk$ instead of $\mathcal{M}_{\Spec \kk}$.

\begin{remark}
  \label{rem:compare-Grothendieck-rings}
  Note that the Grothendieck ring $K_0(\Var_S)$ defined here is
  canonically isomorphic to the Grothendieck ring defined in
  \cite[3.1]{nicaise-sebag-grothendieck-ring-of-varieties}, by
  \cite[3.2.2]{nicaise-sebag-grothendieck-ring-of-varieties}.
\end{remark}

\subsubsection{Pullback}
\label{sec:pullback}

Let $f \colon T \ra S$ be a morphism of $\kk$-varieties.  Then
the functor $\Var_S \ra \Var_T$,
$(X \ra S) \mapsto (|T \times_S X| \ra T \times_S X \ra T)$,
induces a morphism
\begin{equation}
  \label{eq:1}
  f^* \colon K_0(\Var_S) \ra K_0(\Var_T)
\end{equation}
of commutative unital rings
which satisfies $f^*(\DL_S)=\DL_T$
and hence 
induces a morphism
\begin{equation}
  \label{eq:2}
  f^* \colon \mathcal{M}_S \ra \mathcal{M}_T
\end{equation}
of rings. 
If $g \colon U \ra T$ is another morphism of
$\kk$-varieties, we have $g^* f^* =(fg)^*$, by 
\cite[Prop.~4.34]{goertz-wedhorn-AGI}.

In particular, $K_0(\Var_S)$ (resp.\ $\mathcal{M}_S$) becomes a
$K_0(\Var_\kk)$-algebra (resp.\ $\mathcal{M}_\kk$-algebra), and 
\eqref{eq:1} (resp.\ \eqref{eq:2}) is a morphism of
$K_0(\Var_\kk)$-algebras (resp of $\mathcal{M}_\kk$-algebras).
Note that the obvious map defines a canonical
isomorphism  
\begin{equation*}
  \mathcal{M}_\kk \otimes_{K_0(\Var_\kk)} K_0(\Var_S) \sira
  \mathcal{M}_S 
\end{equation*}
of $\mathcal{M}_\kk$-algebras.

\subsubsection{Pushforward}
\label{sec:push-forward}

Let $f \colon T \ra S$ be a morphism of $\kk$-varieties.
The functor $\Var_T \ra \Var_S$, $(Y \xra{y} T) \mapsto (Y
\xra{fy} S)$, induces a morphism
\begin{equation*}
  f_! \colon K_0(\Var_T) \ra K_0(\Var_S)
\end{equation*}
of $K_0(\Var_\kk)$-modules. 
Tensoring with $\mathcal{M}_\kk$ yields a morphism
\begin{equation*}
  f_! \colon \mathcal{M}_T \ra \mathcal{M}_S
\end{equation*}
of $\mathcal{M}_\kk$-modules which sends
$[Y \xra{y} T] \cdot \DL_T^{-i}$ to
$[Y \xra{fy} S] \cdot \DL_S^{-i}$.

\begin{remark}
  \label{rem:compare-Grothendieck-rings-and-pullback}
  The canonical isomorphisms from 
  Remark~\ref{rem:compare-Grothendieck-rings}
  are compatible with pullback and pushforward, 
  by \cite[Prop.~4.34]{goertz-wedhorn-AGI}.
\end{remark}

\subsection{Grothendieck rings of equivariant varieties over a
  base variety}
\label{sec:groth-rings-equiv}

For $n \in \DN_{>0}$ let $\upmu_n=\Spec(\kk[x]/(x^n-1))$ be the
group $\kk$-variety of $n$-th roots of unity. Note that actions
of $\upmu_n$ on a $\kk$-variety $X$ correspond bijectively to group
morphisms $\upmu_n(\kk) \ra \Aut_{\Var_\kk}(X)$. 

Fix a $\kk$-variety $S$ and let $n \in \DN_{>0}$. Recall that a
good $\upmu_n$-action on a $\kk$-variety is a $\upmu_n$-action
such that each $\upmu_n(\kk)$-orbit is contained in an affine
open subset of $X$. An $S$-variety with a good $\upmu_n$-action is
an $S$-variety $p \colon X \ra S$ together with a good
$\upmu_n$-action on $X$. So $p$ is $\upmu_n$-equivariant if we equip
$S$ with the trivial $\upmu_n$-action. We obtain the category
$\Var_S^{\upmu_n}$ of $S$-varieties with good $\upmu_n$-action.
 
The definition of the Grothendieck ring $K_0(\Var_S^{\upmu_n})$ of
$S$-varieties with good $\upmu_n$-action is evident from
\cite[2.2-2.5]{guibert-loeser-merle-convolution}; apart from the
usual scissor relations there is another family of relations,
cf.\ \cite[(2.2.1)]{guibert-loeser-merle-convolution}.  Any
$S$-variety $X \ra S$ with good $\upmu_n$-action gives rise to an
element $[X \ra S]=[X]$ of $K_0(\Var_S^{\upmu_n})$. The
product of $[X \ra S]$ and $[Y \ra S]$ is the element obtained
from $|X \times_S Y| \ra S$ with the obvious diagonal
$\upmu_n$-action.  Define
$\DL_S=\DL_{S, \upmu_n}=[\DA^1_S \ra S] \in K_0(\Var_S^{\upmu_n})$
where $\upmu_n$ acts trivially on $\DA^1_S$. Let
$\mathcal{M}_S^{\upmu_n}:= K_0(\Var_S^{\upmu_n})[\DL_S^{-1}]$.

We write 
$K_0(\Var_\kk^{\upmu_n})$ and $\mathcal{M}_\kk^{\upmu_n}$ instead of 
$K_0(\Var_{\Spec \kk}^{\upmu_n})$ and $\mathcal{M}_{\Spec \kk}^{\upmu_n}$.

If $f \colon T \ra S$ is a morphism of $\kk$-varieties we obtain
as above a pullback morphism
$f^* \colon K_0(\Var_S^{\upmu_n}) \ra K_0(\Var_T^{\upmu_n})$
of $K_0(\Var_\kk^{\upmu_n})$-algebras
satisfying $f^*(\DL_S)=\DL_T$ and an induced
pullback morphism
$f^* \colon \mathcal{M}_S^{\upmu_n} \ra \mathcal{M}_T^{\upmu_n}$
of $\mathcal{M}_\kk^{\upmu_n}$-algebras.
We also have a pushforward morphism
$f_! \colon K_0(\Var_T^{\upmu_n}) \ra K_0(\Var_S^{\upmu_n})$
of $K_0(\Var_\kk^{\upmu_n})$-modules, and a pushforward morphism
$f_! \colon \mathcal{M}_T^{\upmu_n} \ra \mathcal{M}_S^{\upmu_n}$
of $\mathcal{M}_\kk^{\upmu_n}$-modules. 
For $n=1$ we recover the notions from 
\ref{sec:groth-rings-vari}.

Whenever $n'$ is a multiple of $n$ there is a morphism
$\upmu_{n'} \ra \upmu_n$, $\lambda
\mapsto \lambda^{n'/n}$, of $\kk$-group varieties inducing
morphisms 
\begin{align}
  \label{eq:28}
  K_0(\Var_S^{\upmu_n}) & \ra K_0(\Var_S^{\upmu_{n'}}),\\
  \label{eq:29}
  \mathcal{M}_S^{\upmu_n} & \ra \mathcal{M}_S^{\upmu_{n'}},
\end{align}
of rings.
These 
morphism are compatible with pullback and pushforward morphisms.

In particular, $K_0(\Var^{\upmu_n}_S)$ (resp.\
$\mathcal{M}^{\upmu_n}_S$) becomes a $K_0(\Var_\kk)$-algebra
(resp.\ $\mathcal{M}_\kk$-algebra)
and the morphisms 
\eqref{eq:28}
and \eqref{eq:29} are morphisms of algebras. We have a canonical
isomorphism 
\begin{equation}
  \label{eq:15}
  \mathcal{M}_\kk \otimes_{K_0(\Var_\kk)} K_0(\Var^{\upmu_n}_S)
  \sira \mathcal{M}_S^{\upmu_n}  
\end{equation}
of $\mathcal{M}_\kk$-algebras
given by $\tfrac{r}{(\DL_\kk)^n}
\otimes a \mapsto \tfrac{ra}{(\DL_S)^n}$. 

Let $\hat{\upmu}$ be the (inverse) limit of the
$(\upmu_n(\kk))_{n \in \DN_{>0}}$ with respect to the morphisms
$\upmu_{n'}(\kk) \ra \upmu_n(\kk)$,
$\lambda \mapsto \lambda^{n'/n}$, whenever $n'$ is a multiple of
$n$.

An $S$-variety with good $\hat{\upmu}$-action is by definition an
$S$-variety $X \ra S$ together with a group morphism $\hat{\upmu} \ra
\Aut_{\Var_\kk}(X)$ that comes from a good $\upmu_n$-action on $X$,
for some $n \in \DN_{>0}$.
As in \cite[2.2]{guibert-loeser-merle-convolution} we obtain 
the category $\Var_S^{\hat{\upmu}}$ of $S$-varieties with good
$\hat{\upmu}$-action.
We define $K_0(\Var_S^{\hat{\upmu}})$ and
$\mathcal{M}_S^{\hat{\upmu}}$ in the obvious way so that we have
\begin{align*}
  K_0(\Var_S^{\hat{\upmu}}) & = \colim_n K_0(\Var_S^{\upmu_n}),\\
  \mathcal{M}_S^{\hat{\upmu}} & = \colim_n \mathcal{M}_S^{\upmu_n}.
\end{align*}
The Grothendieck ring 
$K_0(\Var_S^{\hat{\upmu}})$
is an $K_0(\Var_\kk)$-algebra (even a 
$K_0(\Var_\kk^{\hat{\upmu}})$-algebra), and  
$\mathcal{M}_S^{\hat{\upmu}}$ is a $\mathcal{M}_\kk$-algebra
(even a 
$\mathcal{M}_\kk^{\hat{\upmu}}$-algebra). We have 
\begin{equation*}
  \mathcal{M}_\kk \otimes_{K_0(\Var_\kk)} K_0(\Var^{\hat{\upmu}}_S)
  \cong \mathcal{M}_S^{\hat{\upmu}}  
\end{equation*}
canonically as rings.
If $f \colon T \ra S$ is a morphism of $\kk$-varieties, we obtain
a pullback morphism
$f^* \colon K_0(\Var_S^{\hat{\upmu}}) \ra  K_0(\Var_T^{\hat{\upmu}})$
of $K_0(\Var_\kk)$-algebras and a pushforward morphism
$f_! \colon
 K_0(\Var_T^{\hat{\upmu}})
\ra
 K_0(\Var_S^{\hat{\upmu}})$
of $K_0(\Var_\kk)$-modules. The base changes of these morphisms along
the ring morphism $K_0(\Var_\kk) \ra \mathcal{M}_\kk$ are denoted
by the same symbols.

Instead of working with $\upmu_n$ we could work more generally with
$\upmu_{n_1} \times \dots \upmu_{n_r}$ (for $r \in \DN$ and $n_1,
\dots, n_r \in 
\DN_{>0}$), and instead of $\hat{\upmu}$ we could work
with $\hat{\upmu}^r$ (for $r \in \DN$). We extend our notation
accordingly. 

\begin{remark}
  \label{rem:dictionary}
  There is an alternative description of 
  $K_0(\Var_S^{\hat{\upmu}^r})$ and
  $\mathcal{M}_S^{\hat{\upmu}^r}$,
  see 
  the dictionary in
  \cite[2.3-2.6]{guibert-loeser-merle-convolution}. 
  When referring to results of
  \cite{guibert-loeser-merle-convolution} we will usually
  translate 
  them using this dictionary.
\end{remark}

\begin{lemma}
  \label{l:open-closed-decomposition}
  Let $S$ be a $\kk$-variety and $F \subset S$ a closed reduced
  subscheme with open complement $U$.
  Let $i \colon F \ra S$ and $j \colon U \ra S$ denote the
  inclusions. Then 
  \begin{align*}
    (j^*, i^*) \colon K_0(\Var_S^{\hat{\upmu}}) 
    & \sira
    K_0(\Var_U^{\hat{\upmu}})
    \times
    K_0(\Var_F^{\hat{\upmu}}),\\
    A & \mapsto (j^*(A), i^*(A)),
  \end{align*}
  is an isomorphism of $K_0(\Var_\kk)$-algebras, with
  inverse given by $(B,C) \mapsto j_!(B) + i_!(C)$.
  Similarly, 
  \begin{equation*}
    (j^*, i^*) \colon \mathcal{M}_S^{\hat{\upmu}} \sira
    \mathcal{M}_U^{\hat{\upmu}}
    \times \mathcal{M}_F^{\hat{\upmu}}
  \end{equation*}
  is an isomorphism of $\mathcal{M}_\kk$-algebras.
\end{lemma}

\begin{proof}
  This is obvious from the definitions.
\end{proof}

\begin{remark}
  \label{rem:standard-multiplication}
  Recall that $K_0(\Var_S)$,
  $K_0(\Var_S^{\upmu_n})$ and
  $K_0(\Var_S^{\hat{\upmu}})$ are $K_0(\Var_\kk)$-algebras
  whose multiplications are
  induced from the fiber product over
  $S$. In the rest of this article mainly the 
  underlying $K_0(\Var_\kk)$-module structure on 
  $K_0(\Var_S)$,
  $K_0(\Var_S^{\upmu_n})$ and
  $K_0(\Var_S^{\hat{\upmu}})$ will be important. 
  Given $(T \ra \Spec \kk)$ in $\Var_\kk$ and $(Z \ra S)$ in $\Var_S$
  or $\Var_S^{\upmu_n}$ or
  $\Var_S^{\hat{\upmu}}$ it is given by
  \begin{equation*}
    [T \ra \Spec \kk].[Z \ra S] = [T \times Z \ra S].
  \end{equation*}
  In fact, we will introduce other multiplications on 
  the 
  $K_0(\Var_\kk)$-modules
  $K_0(\Var_S^{\upmu_n})$ and
  $K_0(\Var_S^{\hat{\upmu}})$
  turning them into 
  $K_0(\Var_\kk)$-algebras.
\end{remark}

\subsection{Convolution}
\label{sec:convolution}

After some preparations we define the convolution product $*$ on
$K_0(\Var_S^{\upmu_n})$ (Definition~\ref{d:convolution})
and show that it turns
$K_0(\Var_S^{\upmu_n})$ into a 
$K_0(\Var_\kk)$-algebra
(Proposition~\ref{p:convolution-comm-ass-unit}). 
This is not a new result: see 
\cite[5.1-5.5]{guibert-loeser-merle-convolution} and use the
dictionary from Remark~\ref{rem:dictionary}. 
Nevertheless we liked the exercise of showing associativity 
without using this dictionary.



Let $S$ be a $\kk$-variety and $n \in \DN_{>0}$.  Let
$p \colon Z \ra S$ be an object of $\Var^{\upmu_n \times \upmu_n}_S$.
We assume that $\upmu_n \times \upmu_n$ acts on $Z$ from the right.
The group $\upmu_n \times \upmu_n$ acts on the $\kk$-variety
$Z \times \DGm \times \DGm$ via
$(z,x,y).(s,t):=(z.(s,t), s^{-1} x, t^{-1} y)$. The quotient with
respect to this action is the balanced
product 
$Z \times^{\upmu_n \times \upmu_n} \DGm \times \DGm$ which is again
a $\kk$-variety (use \cite[Exp.~V.1]{SGA-1}).
We equip it with the
diagonal $\upmu_n$-action given by 
$[z,x,y].t= [z, tx, ty]=[z.(t,t), x, y]$.
With the obvious morphism to $S$ induced by $p$ it is 
an object of
$\Var_S^{\upmu_n}$.
Similarly, starting from the two closed $S$-subvarieties
of $Z \times \DGm \times \DGm$ defined by the equations $x^n+y^n=1$ and
$x^n+y^n=0$, we obtain the two objects
$(Z \times^{\upmu_n \times \upmu_n} \DGm \times \DGm)|_{x^n+y^n=1}$
and
$(Z \times^{\upmu_n \times \upmu_n} \DGm \times \DGm)|_{x^n+y^n=0}$
of $\Var_S^{\upmu_n}$.

Given $(Z \xra{p} S) \in \Var_S^{\upmu_n \times \upmu_n}$ as above define
\begin{align}
  \label{eq:psi-two-mu-s}
  \Psi(Z \xra{p} S)
  := &
  -[(Z \times^{\upmu_n \times \upmu_n} \underset{x}{\DGm} \times
  \underset{y}{\DGm})|_{x^n+y^n=1} \xra{[z,x,y] \mapsto p(z)}
  S]\\
  \notag
  & +[(Z \times^{\upmu_n \times \upmu_n} \underset{x}{\DGm} \times
  \underset{y}{\DGm})|_{x^n+y^n=0}
  \xra{([z,x,y]) \mapsto p(z)} S] 
  \in K_0(\Var_S^{\upmu_n}).
\end{align}
Here the symbols $x$ and $y$ below $\DGm \times \DGm$ indicate
that $(x,y)$ forms a system 
of coordinates 
on $\DGm \times \DGm$.
Similar notation will be
used below without further explanations.

\begin{example}
  \label{ex:trivial-action}
  Let $p \colon Z \ra S$ be as above and assume that
  $\upmu_n \times \upmu_n$ acts trivially on $Z$. Then
  $Z \times^{\upmu_n \times \upmu_n} \DGm \times \DGm \sira Z
  \times \DGm \times \DGm$,
  $[z,x,y] \mapsto (z,x^n, y^n)$, is an isomorphism which is
  $\upmu_n$-equivariant if we equip 
  $Z \times \DGm \times \DGm$
  with the trivial $\upmu_n$-action. This implies
  $\Psi(Z \xra{p} S) = [Z \xra{p} S]$ in $K_0(\Var_S^{\upmu_n})$
  where $Z$ is considered as a $\upmu_n$-variety over $S$ with
  trivial action. In particular, we obtain
  $\Psi(S \xra{\id} S) = [S \xra{\id} S]$.
\end{example}

\begin{example}
  \label{ex:p-product-over-k-trivial-on-one-factor}
  Assume that $p=p_1 \times p_2 \colon Z=Z_1 \times Z_2 \ra S=S_1
  \times S_2$ where $S_1$ and $S_2$ are $\kk$-varieties and 
  $p_i \colon Z_i \ra S_i$
  is an object of $\Var_{S_i}^{\upmu_n}$, for $i=1,2$.
  Moreover assume that the action of $\upmu_n$ on
  $Z_2$ is trivial. Then 
  $Z_1 \times Z_2 \times^{\upmu_n \times \upmu_n} \DGm \times
  \DGm \sira (Z_1 \times^{\upmu_n} \DGm) \times (Z_2 \times \DGm)$,
  $[z_1,z_2,x,y] \mapsto ([z_1,x], z_2, y^n)$, is an isomorphism
  over $S$, and we can simplify 
  \eqref{eq:psi-two-mu-s} to
  \begin{align*}
    \Psi(Z_1 \times Z_2 \xra{p_1 \times p_2} S_1 \times S_2)
    = 
    &
      -[(Z_1 \times^{\upmu_n} \underset{x}{\DGm})|_{x^n\not=1} \times Z_2
      \ra S_1 \times S_2]\\
    \notag
    &
      +[(Z_1 \times^{\upmu_n} \underset{x}{\DGm}) \times Z_2
      \ra S_1 \times S_2]\\
    \notag
    =
    & [(Z_1 \times^{\upmu_n} \upmu_n) \times Z_2
      \ra S_1 \times S_2]\\
    \notag
    =
    & [Z_1 \times Z_2
      \ra S_1 \times S_2].
  \end{align*}
  This example will be useful later on.
\end{example}

In fact, $\Psi$ induces a morphism
\begin{equation}
  \label{eq:10}
  \Psi \colon  K_0(\Var_S^{\upmu_n \times \upmu_n}) \ra K_0(\Var_S^{\upmu_n})
\end{equation}
of $K_0(\Var_\kk)$-modules.

Our next aim is to prove Proposition~\ref{p:Psi-associative}
which will later on imply associativity of the convolution
product.

Let $p \colon Z \ra S$ be an object of $\Var^{\upmu_n \times \upmu_n
  \times \upmu_n}_S$. Similarly as above we 
define
\begin{multline}
  \label{eq:psi-123}
  \Psi_{123}(Z \xra{p} S) := 
  -[(Z \times^{\upmu_n \times \upmu_n \times
    \upmu_n} \underset{x_1}{\DGm} \times \underset{x_2}{\DGm}
    \times \underset{x_3}{\DGm})|_{x_1^n+x_2^n+x_3^n=1}
    \xra{[z,x_1,x_2,x_3] \mapsto p(z)}
    S]\\
    +[(Z
    \times^{\upmu_n \times \upmu_n \times
    \upmu_n} \underset{x_1}{\DGm} \times \underset{x_2}{\DGm}
    \times \underset{x_3}{\DGm})|_{x_1^n+x_2^n+x_3^n=0}
    \xra{([z,x_1,x_2,x_3]) \mapsto p(z)} S] \in K_0(\Var_S^{\upmu_n})
\end{multline}
where the closed subvarieties of $Z \times^{\upmu_n \times \upmu_n \times \upmu_n}
\underset{x_1}{\DGm} \times \underset{x_2}{\DGm} \times
\underset{x_3}{\DGm}$ are equipped with the $\upmu_n$-action
$[z,x_1,x_2,x_3].t=[z,tx_1, tx_2, tx_3]=[z.(t,t,t), x_1, x_2,
x_3]$.
Again we obtain a morphism
\begin{equation*}
  \Psi_{123} \colon  K_0(\Var_S^{\upmu_n \times \upmu_n \times \upmu_n}) \ra
  K_0(\Var_S^{\upmu_n}) 
\end{equation*}
of $K_0(\Var_k)$-modules.

Similarly we associate to 
$(p \colon Z \ra S) \in \Var^{\upmu_n \times \upmu_n
  \times \upmu_n}_S$ the element
\begin{align}
  \label{eq:psi-13}
  \Psi_{13}&(Z) := 
  -[(Z \times^{\upmu_n \times \{1\} \times
    \upmu_n} \underset{x_1}{\DGm} \times \{1\}
    \times \underset{x_3}{\DGm})|_{x_1^n+x_3^n=1}
    \xra{[z,x_1,1,x_3] \mapsto p(z)}
    S]
  \\
  \notag 
  & +[(Z
    \times^{\upmu_n \times \{1\} \times
    \upmu_n} \underset{x_1}{\DGm} \times \{1\}
    \times \underset{x_3}{\DGm})|_{x_1^n+x_3^n=0}
    \xra{([z,x_1,1,x_3]) \mapsto p(z)} S] 
    \in K_0(\Var_S^{\upmu_n \times \upmu_n}).
\end{align}
Here the $\upmu_n \times \upmu_n$-action is
given by the two commuting $\upmu_n$-actions
$[z,x_1,1,x_3].s=[z,sx_1, 1, sx_3]=[z.(s,1,s), x_1, 1, x_3]$
and
$[z,x_1,1,x_3].t=[z.(1,t,1),x_1, 1, x_3]$, i.\,e.\ we have
$[z,x_1,1,x_3].(s,t) =[z.(s,t,s), x_1, 1, x_3]$.
As above we obtain a morphism
\begin{equation*}
  \Psi_{13} \colon  K_0(\Var_S^{\upmu_n \times \upmu_n \times \upmu_n}) \ra
  K_0(\Var_S^{\upmu_n \times \upmu_n}) 
\end{equation*}
of $K_0(\Var_k)$-modules.
Similarly we define $\Psi_{12}$ and $\Psi_{23}$.

\begin{remark}
  \label{rem:Psi-compatible-with-pushforward-pullback}
  If $f \colon  S \ra S'$ is a morphism of $\kk$-varieties, all maps
  $\Psi$, 
  $\Psi_{123}$, 
  $\Psi_{12}$,
  $\Psi_{13}$,
  $\Psi_{23}$
  are compatible with $f_!$ and $f^*,$ for example
  $\Psi(f_!(Z))=f_!(\Psi(Z))$ for $Z \in
  K_0(\Var_{S}^{\upmu_n \times \upmu_n})$
  and  
  $\Psi(f^*(Z))=f^*(\Psi(Z))$ for $Z \in
  K_0(\Var_{S'}^{\upmu_n \times \upmu_n})$.
  For $f_!$ this is obvious. For $f^*$ one uses the fact that
  $Z \times \DGm \times \DGm \ra Z \times^{\upmu_n \times
    \upmu_n} \DGm \times \DGm$ is a $(\upmu_n \times
  \upmu_n)$-torsor and hence its pullback under the base change
  morphism $f$ is again such a torsor.
\end{remark}

\begin{proposition}
  [{\cite[Prop.~5.5]{guibert-loeser-merle-convolution}}]
  \label{p:Psi-associative}
  We have
  \begin{equation*}
    \Psi_{123}=\Psi \circ \Psi_{13}
    =\Psi \circ \Psi_{12}
    =\Psi \circ \Psi_{23}
  \end{equation*}
  as morphisms 
  $K_0(\Var_S^{\upmu_n \times \upmu_n \times \upmu_n}) \ra
  K_0(\Var_S^{\upmu_n})$
  of $K_0(\Var_\kk)$-modules.
\end{proposition}

\begin{proof}
  Let $p \colon Z \ra S$ be an object of
  $\Var^{\upmu_n \times \upmu_n \times \upmu_n}_S$. It is enough to
  show that
  $\Psi_{123}(Z) =\Psi(\Psi_{13}(Z)) =\Psi(\Psi_{12}(Z))
  =\Psi(\Psi_{23}(Z))$
  in $K_0(\Var_S^{\upmu_n})$. We only prove
  $\Psi_{123}(Z) =\Psi(\Psi_{13}(Z))$ and leave the remaining
  cases to the reader.
  
  From 
  \eqref{eq:psi-two-mu-s}
  and 
  \eqref{eq:psi-13}
  we obtain
  \begin{align}
    \label{eq:psi-psi13}
    \Psi(\Psi_{13}(Z)) = 
    & -\Psi([(Z \times^{\upmu_n \times \{1\}
      \times \upmu_n} \underset{x_1}{\DGm} \times \{1\} \times
    \underset{x_3}{\DGm})|_{x_1^n+x_3^n=1} \ra
    S])\\
    \notag 
    & +\Psi([(Z \times^{\upmu_n \times \{1\} \times \upmu_n}
      \underset{x_1}{\DGm} \times \{1\} \times
      \underset{x_3}{\DGm})|_{x_1^n+x_3^n=0} \ra S])\\
    \notag = &
    \sum_{\delta, \epsilon \in \{0,1\}}
    (-1)^{\delta+\epsilon}
    [D_{\delta, \epsilon} \ra S]
  \end{align}
  where 
  \begin{align*}
    D_{\delta, \epsilon} & :=
    ((Z \times^{\upmu_n \times \{1\} \times \upmu_n}
    \underset{x_1}{\DGm} \times \{1\} \times
    \underset{x_3}{\DGm})|_{x_1^n+x_3^n=\delta} \times^{\upmu_n \times
      \upmu_n} \underset{y_1}{\DGm} \times
    \underset{y_2}{\DGm})|_{y_1^n+y_2^n=\epsilon}\\
    \notag
    & =
    (Z \times^{\upmu_n \times \upmu_n}
    \underset{x_1}{\DGm} \times
    \underset{x_3}{\DGm} \times^{\upmu_n \times
      \upmu_n} \underset{y_1}{\DGm} \times
    \underset{y_2}{\DGm})|_{x_1^n+x_3^n=\delta, \atop
      y_1^n+y_2^n=\epsilon}\\
    \notag
    & =
    \frac{(Z \times
      \underset{x_1}{\DGm} \times
      \underset{x_3}{\DGm} \times \underset{y_1}{\DGm} \times
      \underset{y_2}{\DGm})|_{x_1^n+x_3^n=\delta, \atop
        y_1^n+y_2^n=\epsilon}}
    {\upmu_n \times \upmu_n \times \upmu_n \times \upmu_n}.
  \end{align*}
  Here, by the definitions of the quotients in
  \eqref{eq:psi-two-mu-s}
  and \eqref{eq:psi-13},
  the quotient is formed with respect to the
  $(\upmu_n)^{\times 4}$-action 
  \begin{equation*}
    (z,x_1, x_3, y_1, y_2).(s,t,u,v)=
    (z.(su,v,tu),s^{-1}x_1, t^{-1}x_3, u^{-1}y_1, v^{-1}y_2).    
  \end{equation*}
  and $D_{\delta,\epsilon}$ is a $\upmu_n$-variety with action
  \begin{equation*}
    [z,x_1,x_3,y_1,y_2].m=[z,x_1,x_3,my_1,my_2]=
    [z.(m,m,m),x_1,x_3,y_1,y_2].
  \end{equation*}

  The coordinate changes $a_1=x_1y_1$, $a_2=x_3y_1$, $b=y_1$,
  $a_3=y_2$ in $(\DGm)^{\times 4}$ and $s'=su$,
  $t'=tu$, $u=u$, $v=v$ in $(\upmu_n)^{\times 4}$ show that
  \begin{equation*}
    D_{\delta,\epsilon}
    \cong
    \frac{(Z \times
      \underset{a_1}{\DGm} \times
      \underset{a_2}{\DGm} \times \underset{b}{\DGm} \times
      \underset{a_3}{\DGm})|_{a_1^n+a_2^n=\delta b^n, \atop
        b^n+a_3^n=\epsilon}}
    {\upmu_n \times \upmu_n \times \upmu_n \times \upmu_n}.
  \end{equation*}
  where the quotient is formed with respect to the
  $(\upmu_n)^{\times 4}$-action 
  \begin{equation*}
    (z,a_1, a_2, b, a_3).(s',t',u,v)=
    (z.(s',v,t'),s'^{-1}a_1, t'^{-1}a_2, u^{-1}b, v^{-1}a_3)    
  \end{equation*}
  and the $\upmu_n$-action on this quotient is given by 
  \begin{equation*}
    [z,a_1,a_2,b,a_3].m=[z,ma_1,ma_2,mb,ma_3]=
    [z.(m,m,m),a_1,a_2,b,a_3].
  \end{equation*}
  The quotient of 
  $\underset{a_1}{\DGm} \times
  \underset{a_2}{\DGm} \times 
  \underset{b}{\DGm} \times 
  \underset{a_3}{\DGm}|_{a_1^n+a_2^n=\delta b^n, \atop
        b^n+a_3^n=\epsilon}$
  under the obvious action of $\{1\} \times \{1\} \times \upmu_n
  \times \{1\}$ on the factor $\DGm$ with coordinate $b$ is 
  clearly isomorphic to 
  \begin{equation*}
    Q_{\delta, \epsilon}:=(\underset{a_1}{\DGm} \times
    \underset{a_2}{\DGm} \times 
    \underset{a_3}{\DGm})|_{a_1^n+a_2^n=\delta (\epsilon -a_3^n),
      \atop 
      a_3^n\not=\epsilon}
  \end{equation*}
  So we obtain
  \begin{equation*}
    D_{\delta, \epsilon} \cong Z \times^{\upmu_n \times \upmu_n \times
      \upmu_n} Q_{\delta,\epsilon} 
  \end{equation*}
  where the quotient is formed with respect
  to the $(\upmu_n)^{\times 3}$-action
  \begin{equation*}
    (z,a_1, a_2, a_3).(s',t',v)=
    (z.(s',v,t'),s'^{-1}a_1, t'^{-1}a_2, v^{-1}a_3)    
  \end{equation*}
  and the $\upmu_n$-action on this quotient is given by 
  \begin{equation*}
    [z,a_1,a_2,a_3].m=
    [z.(m,m,m),a_1,a_2,a_3]=
    [z,ma_1,ma_2,ma_3].
  \end{equation*}
  Continuing the computation
  \eqref{eq:psi-psi13} we obtain
  \begin{align*}
    \Psi(\Psi_{13}(Z)) 
    = &
    + [Z \times^{\upmu_n \times \upmu_n
      \times \upmu_n}
    (\underset{a_1}{\DGm} \times
    \underset{a_2}{\DGm} \times
    \underset{a_3}{\DGm})|_{a_1^n+a_2^n+a_3^n=1
      \atop 1 \not= a_3^n}
    \ra S]\\
    \notag
    & -[Z \times^{\upmu_n \times \upmu_n
      \times \upmu_n}
    (\underset{a_1}{\DGm} \times
      \underset{a_2}{\DGm} \times
      \underset{a_3}{\DGm})|_{a_1^n+a_2^n+a_3^n=0}
    \ra S]\\
    \notag
    & -[Z \times^{\upmu_n \times \upmu_n
      \times \upmu_n}
    (\underset{a_1}{\DGm} \times
    \underset{a_2}{\DGm} \times \underset{a_3}{\DGm})|_{a_1^n+a_2^n=0
      \atop 1 \not= a_3^n}
    \ra S]\\
    \notag
    & + [Z \times^{\upmu_n \times \upmu_n
      \times \upmu_n}
    (\underset{a_1}{\DGm} \times
    \underset{a_2}{\DGm} \times \underset{a_3}{\DGm})|_{a_1^n+a_2^n=0}
    \ra S]
  \end{align*}
  The last two summands simplify to 
  \begin{equation*}
    +[Z \times^{\upmu_n \times \upmu_n
      \times \upmu_n}
    (\underset{a_1}{\DGm} \times
    \underset{a_2}{\DGm} \times
    \underset{a_3}{\DGm})|_{a_1^n+a_2^n=0 
      \atop 1=a_3^n} 
    \ra S].
  \end{equation*}
  The two conditions
  $a_1^n+a_2^n=0$ and $1=a_3^n$ are equivalent to
  the two conditions
  $a_1^n+a_2^n+a_3^n=1$ and $1=a_3^n$.
  Hence we can further simplify and obtain
  \begin{align*}
    \Psi(\Psi_{13}(Z)) = &
    + [Z \times^{\upmu_n \times \upmu_n
      \times \upmu_n}
    (\underset{a_1}{\DGm} \times
    \underset{a_2}{\DGm} \times
    \underset{a_3}{\DGm})|_{a_1^n+a_2^n+a_3^n=1}
    \ra S]\\
    & -[Z \times^{\upmu_n \times \upmu_n
      \times \upmu_n}
    (\underset{a_1}{\DGm} \times
      \underset{a_2}{\DGm} \times
      \underset{a_3}{\DGm})|_{a_1^n+a_2^n+a_3^n=0}
    \ra S]\\
    = &
    \Psi_{123}(Z).
  \end{align*}
  where the last equality holds by definition
  \eqref{eq:psi-123}.
\end{proof}

\begin{definition}
  [{Convolution product}]
  \label{d:convolution}
  The convolution product $*$ on $K_0(\Var_S^{\upmu_n})$ is
  defined as the $K_0(\Var_\kk)$-linear composition
  \begin{equation}
    \label{eq:25}
    * \colon K_0(\Var_S^{\upmu_n}) \otimes_{K_0(\Var_\kk)}
    K_0(\Var_S^{\upmu_n})  
    \xra{\times_S} K_0(\Var_S^{\upmu_n \times \upmu_n})
    \xra{\Psi}  K_0(\Var_S^{\upmu_n})
  \end{equation}
  where the first map $\times_S$ is 
  the $K_0(\Var_\kk)$-linear map
  induced by mapping a pair
  $(A,B)$ of 
  $S$-varieties with good $\upmu_n$-action to the class of the
  $S$-variety $|A \times_S B|$ with good $(\upmu_n \times
  \upmu_n)$-action. 
\end{definition}

More explicitly, if $A \ra S$ and $B \ra S$ are
$S$-varieties with good $\upmu_n$-action, then
\begin{align}
  \label{eq:convolution-explicit}
  [A \ra S] * [B \ra S] = 
  & -[(|A \times_S B| \times^{\upmu_n \times \upmu_n}
    \underset{x}{\DGm} \times 
    \underset{y}{\DGm})|_{x^n+y^n=1} \ra S]\\
  \notag
  & +[(|A \times_S B| \times^{\upmu_n \times \upmu_n}
    \underset{x}{\DGm} \times 
    \underset{y}{\DGm})|_{x^n+y^n=0}
    \ra S]
  \\
  \notag
  =
  & -[|(A \times_S B \times^{\upmu_n \times \upmu_n}
    \underset{x}{\DGm} \times 
    \underset{y}{\DGm})|_{x^n+y^n=1}| \ra S]
  \\
  \notag
  & +[|(A \times_S B \times^{\upmu_n \times \upmu_n}
    \underset{x}{\DGm} \times 
    \underset{y}{\DGm})|_{x^n+y^n=0}|
    \ra S]. 
\end{align}
The second equality comes from the fact that taking the reduced
subscheme structure commutes with fiber 
products 
(\cite[Prop.~4.34]{goertz-wedhorn-AGI})
and with quotients under the action of a finite group.

\begin{remark}
  \label{rem:action-on-B-trivial}
  Let $A \ra S$ and $B \ra S$ be
  $S$-varieties with good $\upmu_n$-action, and assume that the
  $\upmu_n$-action on $B$ is trivial.
  Similar as in
  Example~\ref{ex:p-product-over-k-trivial-on-one-factor}
  we deduce from \eqref{eq:convolution-explicit} that
  \begin{multline*}
    [A] * [B]
    = 
    -[|((A \times^{\upmu_n}
    \underset{x}{\DGm}) \times_S 
    (B \times \underset{y'}{\DGm}))|_{x^n+y'=1}|]
    +[|((A \times^{\upmu_n}
    \underset{x}{\DGm}) \times_S 
    (B \times \underset{y'}{\DGm}))|_{x^n+y'=0}|]
    \\
    = 
    -[|((A \times^{\upmu_n}
    \underset{x}{\DGm}) \times_S B)|_{x^n\not=1}|]
    +[|(A \times^{\upmu_n} \DGm) \times_S B|]
    \\
    = [|(A \times^{\upmu_n} \upmu_n) \times_S B|]
    = [|A \times_S B|]= [A] [B].
  \end{multline*}
\end{remark}

\begin{proposition}
  [{\cite[Prop.~5.2]{guibert-loeser-merle-convolution}}]
  \label{p:convolution-comm-ass-unit}
  Let $S$ be a $\kk$-variety and $n \geq 1$.
  The convolution product $*$ turns $K_0(\Var_S^{\upmu_n})$ into
  an associative commutative unital $K_0(\Var_\kk)$-algebra. The
  identity element is  
  the class of $(\id_S \colon S \ra S)$ where $\upmu_n$ acts
  trivially on $S$.  
  We denote this ring as $(K_0(\Var_S^{\upmu_n}),*)$.
\end{proposition}

\begin{proof}
  Clearly, the convolution product is commutative.
  Remark~\ref{rem:action-on-B-trivial}
  shows that $[\id_S \colon S \ra S]$ is the identity with respect to
  the convolution 
  product.  
  Associativity follows from 
  Proposition~\ref{p:Psi-associative}:
  \begin{multline*}
    ([A] * [B]) * [C] 
    = \Psi(\Psi_{12}([|A \times_S B \times_S C|]) 
    = \Psi(\Psi_{23}([|A \times_S B \times_S C|]) 
    = [A] * ([B] * [C]).
  \end{multline*}
  Here we again use that passing to the reduced subscheme structure
  commutes with fiber products and taking quotients under the
  action of a finite group.
\end{proof}

\begin{remark}
  \label{rem:convolution-n=1}
  For $n=1$ the convolution product $*$ on $K_0(\Var_S^{\upmu_1})$
  coincides with the product on
  $K_0(\Var_S)=K_0(\Var_S^{\upmu_1})$,
  so $K_0(\Var_S)=(K_0(\Var_S^{\upmu_1}),*)$ as $K_0(\Var_\kk)$-algebras.
  This follows immediately from 
  Remark~\ref{rem:action-on-B-trivial}.
\end{remark}

Let $(Z \ra S) \in \Var_S^{\upmu_n \times \upmu_n}$ and assume
that $n'=dn$ is a multiple of $n$. Then the morphism
$Z \times \DGm \times \DGm \ra Z \times \DGm \times \DGm$,
$(z,x,y) \mapsto (z,x^d,y^d)$ defines an isomorphism
\begin{equation}
  \label{eq:22}
  Z \times^{\upmu_{n'} \times \upmu_{n'}} \DGm \times \DGm \sira 
  Z \times^{\upmu_n \times \upmu_n} \DGm \times \DGm
\end{equation}
in $\Var_S^{\upmu_{n'}}$.
This implies that $\Psi$ is compatible with the morphisms
$K_0(\Var_S^{\upmu_n \times \upmu_n}) \ra K_0(\Var_S^{\upmu_{n'}
  \times \upmu_{n'}})$
and 
$K_0(\Var_S^{\upmu_n}) \ra K_0(\Var_S^{\upmu_{n'}})$,
cf.\ \eqref{eq:28}, and so is the 
first map in \eqref{eq:25}.
We deduce that the obvious morphism
\begin{equation*}
  (K_0(\Var_S^{\upmu_n}),*) \ra (K_0(\Var_S^{\upmu_{n'}}),*)
\end{equation*}
is a map of $K_0(\Var_\kk)$-algebras.
Hence convolution turns $K_0(\Var_S^{\hat{\upmu}})$ 
into an associative commutative unital 
$K_0(\Var_\kk)$-algebra; we denote this algebra by 
$(K_0(\Var_S^{\hat{\upmu}}),*)$. 
 
If $f \colon T \ra S$ is a morphism of $\kk$-varieties, 
the 
pullback maps 
$f^* \colon (K_0(\Var_S^{\upmu_n}),*) \ra
(K_0(\Var_T^{\upmu_n}),*)$
and
$f^* \colon (K_0(\Var_S^{\hat{\upmu}}),*) \ra
(K_0(\Var_T^{\hat{\upmu}}),*)$
are maps of $K_0(\Var_\kk)$-algebras (use
Remark~\ref{rem:Psi-compatible-with-pushforward-pullback}
and that the first map in \eqref{eq:25} is compatible
with pullbacks). 


We also want to define a convolution product on
$\mathcal{M}_S^{\upmu_n}$ and $\mathcal{M}_S^{\hat{\upmu}}$.

Consider 
the localization of  
$(K_0(\Var_S^{\upmu_n}),*)$ at the multiplicative set $\{1,
\DL_S, \DL_S * \DL_S, \dots\}$. The $n$-fold convolution
product of $\DL_S=[\DA^1_S]$ with itself is $[\DA^n_S]$ 
and we have 
$[A]*[\DA^n_S]=[A][\DA^n_S]$ for $[A] \in K_0(\Var_S^{\upmu_n})$,
by
Remark~\ref{rem:action-on-B-trivial}. Hence the underlying
abelian group of
this localization is canonically identified with the underlying
abelian group of $\mathcal{M}_S^{\upmu_n}$. We can therefore
denote the above localization by $(\mathcal{M}_S^{\upmu_n},*)$.

Because the structure morphism $K_0(\Var_\kk) \ra
(K_0(\Var_S^{\upmu_n}),*)$ sends $\DL_\kk$ to $\DL_S$ we obtain
a canonical isomorphism
\begin{equation}
  \label{eq:26}
  \mathcal{M}_\kk \otimes_{K_0(\Var_\kk)} (K_0(\Var_S^{\upmu_n}),*)
  \sira
  (\mathcal{M}_S^{\upmu_n},*)
\end{equation}
of $\mathcal{M}_\kk$-algebras which we will often treat as an
equality in the following. Its underlying morphism of 
$\mathcal{M}_\kk$-modules coincides with \eqref{eq:15}.

Similarly, we define the convolution product $*$ on
$\mathcal{M}_S^{\hat{\upmu}}$ and obtain the $\mathcal{M}_\kk$-algebra
$\mathcal{M}_\kk \otimes_{K_0(\Var_\kk)} (K_0(\Var_S^{\hat{\upmu}}),*)=
(\mathcal{M}_S^{\hat{\upmu}},*)$. 
The map $\Psi$ from \eqref{eq:10} gives in the obvious way rise to a morphism
\begin{equation*}
  \Psi 
  \colon  
  \mathcal{M}_S^{\hat{\upmu} \times \hat{\upmu}} 
  \ra 
  \mathcal{M}_S^{\hat{\upmu}}
\end{equation*}
of $\mathcal{M}_\kk$-modules; 
the 
convolution product $*$ on
$\mathcal{M}_S^{\hat{\upmu}}$ is then given as the composition
\begin{equation*}
  * \colon \mathcal{M}_S^{\hat{\upmu}} \otimes_{\mathcal{M}_\kk}
  \mathcal{M}_S^{\hat{\upmu}}
  \xra{\times_S} \mathcal{M}_S^{\hat{\upmu} \times \hat{\upmu}}
  \xra{\Psi} \mathcal{M}_S^{\hat{\upmu}}.
\end{equation*}

Given $f \colon T \ra S$ as above, we obtain pullback maps
$f^* \colon (\mathcal{M}_S^{\upmu_n},*) \ra (\mathcal{M}_T^{\upmu_n},*)$
and
$f^* \colon (\mathcal{M}_S^{\hat{\upmu}},*) \ra
(\mathcal{M}_T^{\hat{\upmu}},*)$ of
$\mathcal{M}_\kk$-algebras. 
Under the isomorphisms \eqref{eq:26} they are just obtained by
scalar extension along $K_0(\Var_\kk) \ra \mathcal{M}_\kk$ from
the previous pullback maps.

\subsection{Convolution of varieties over
  \texorpdfstring{$\DA^1_\kk$}{A1}}
\label{sec:conv-vari-over}

We now use that $\DA^1_\kk$ is a commutative 
group
$\kk$-variety.
Let $\add \colon  \DA^1_\kk \times \DA^1_\kk \ra \DA^1_\kk$, $(x,y)
\mapsto x+y$, be the 
addition morphism. Let $n \geq 1$.

\begin{definition}
  [{Convolution over $\DA^1_\kk$}]
  \label{d:convolution-over-A1}
  The convolution product $\star$ on $K_0(\Var_{\DA^1_\kk}^{\upmu_n})$ is
  defined as the $K_0(\Var_\kk)$-linear composition
  \begin{equation}
    \label{eq:convolution-A1}
    \star \colon K_0(\Var_{\DA^1_\kk}^{\upmu_n}) \otimes_{K_0(\Var_\kk)}
    K_0(\Var_{\DA^1_\kk}^{\upmu_n})  
    \xra{\times}
    K_0(\Var_{\DA^1_\kk \times \DA^1_\kk}^{\upmu_n \times \upmu_n})
    \xra{\add_!}
    K_0(\Var_{\DA^1_\kk}^{\upmu_n \times \upmu_n})
    \xra{\Psi}  K_0(\Var_{\DA^1_\kk}^{\upmu_n})
  \end{equation}
  where the first map $\times$ is 
  the $K_0(\Var_\kk)$-linear map
  induced by mapping a pair
  $(A,B)$ of 
  $\DA^1_\kk$-varieties with good $\upmu_n$-action to the
  $S$-variety $A \times B$ with good $(\upmu_n \times
  \upmu_n)$-action. 
\end{definition}

By Remark~\ref{rem:Psi-compatible-with-pushforward-pullback} we
have
\begin{equation*}
  A \star B
  =\Psi(\add_!(A \times B)) 
  =\add_!(\Psi(A \times B))
\end{equation*}
for $A$, $B \in 
K_0(\Var_{\DA^1_\kk}^{\upmu_n})$.


\begin{remark}
  \label{rem:action-on-B-trivial-over-A1}
  Let $A \xra{\alpha} \DA^1_\kk$ and $B \xra{\beta} \DA^1_\kk$ be
  $\DA^1_\kk$-varieties with good $\upmu_n$-action, and assume
  that the 
  $\upmu_n$-action on $B$ is trivial.
  Then 
  Example~\ref{ex:p-product-over-k-trivial-on-one-factor}
  implies that
  \begin{equation*}
    [A \xra{\alpha} \DA^1_\kk] \star [B \xra{\beta} \DA^1_\kk] =
    [A \times B \xra{\alpha \circledast \beta} \DA^1_\kk]
  \end{equation*}
  where $(\alpha \circledast \beta)(a,b)=\alpha(a)+\beta(b)$;
  the $\upmu_n$-action on $A \times B$ is the obvious diagonal
  action $(a,b).t=(a.t,b.t)=(a.t,b)$.
\end{remark}

\begin{remark}
  \label{rem:no-group-action}
  In the case $n=1$ the convolution product $\star$ on
  $K_0(\Var_{\DA^1_\kk}^{\upmu_1})= K_0(\Var_{\DA^1_\kk})$
  satisfies
  \begin{equation*}
    [A \xra{\alpha} \DA^1_\kk] \star [B \xra{\beta} \DA^1_\kk] =
    [A \times B \xra{\alpha \circledast \beta} \DA^1_\kk]
  \end{equation*}
  where $A \xra{\alpha} \DA^1_\kk$ and $B \xra{\beta} \DA^1_\kk$ 
  are 
  $\DA^1_\kk$-varieties.
  This is a special case of
  Remark~\ref{rem:action-on-B-trivial-over-A1}.   
\end{remark}

\begin{proposition}
  \label{p:A1-convolution-comm-ass-unit}
  The convolution product $\star$ 
  turns 
  $K_0(\Var_{\DA^1_\kk}^{\upmu_n})$
  into an associative commutative $K_0(\Var_\kk)$-algebra with
  identity element 
  $[\Spec k \xra{0} \DA^1_\kk]$.
  We denote this ring by $(K_0(\Var_{\DA^1_\kk}^{\upmu_n}), \star)$.
\end{proposition}

\begin{proof}
  Commutativity follows from commutativity of $\DA^1_\kk$.
  That 
  $[\Spec k \xra{0} \DA^1_\kk]$ is the identity element with respect
  to $\star$ follows from
  Remark~\ref{rem:action-on-B-trivial-over-A1}.
  Denote the morphism $(\DA^1_\kk)^{\times 3} \ra \DA^1_\kk$,
  $(x,y,z) 
  \mapsto x+y+z$ by $\op{addd}$.
  Remark~\ref{rem:Psi-compatible-with-pushforward-pullback} and
  Proposition~\ref{p:Psi-associative}
  show that
  \begin{align*}
    (A \star B) \star C 
    & = 
    \add_!(\Psi(\add_!(\Psi(A \times B)) \times C))\\
    & = 
    \add_!((\add \times \id)_!(\Psi(\Psi(A \times B) \times C)))\\
    & = \op{addd}_!
      (\Psi(\Psi_{12}(A \times B \times C)))\\
    & = 
    \op{addd}_!(\Psi_{123}(A \times B \times C))
  \end{align*}
  A similar computation shows that the last term equals $A
  \star 
  (B \star C)$. This proves associativity.
\end{proof}

Mapping a $\kk$-variety $(A \ra \Spec \kk)$ to $(A \xra{0}
\DA^1_\kk)$ induces a morphism of unital rings
\begin{equation}
  \label{eq:46}
  K_0(\Var_\kk) \ra (K_0(\Var_{\DA^1_\kk}^{\upmu_n}), \star)
\end{equation}
as follows immediately from
Remark~\ref{rem:action-on-B-trivial-over-A1}. 
This map is the structure map of the  $K_0(\Var_\kk)$-algebra
$(K_0(\Var_{\DA^1_\kk}^{\upmu_n}), \star)$.
Denote the image of $\DL_\kk$ under this map by
\begin{equation}
  \label{eq:def-L-A1-0}
  \DL_{(\DA^1_\kk,0)}:=[\DA^1_\kk \xra{0} \DA^1_\kk] \in
  K_0(\Var_{\DA^1_\kk}^{\upmu_n}). 
\end{equation}

Let us denote the localization of 
$(K_0(\Var_{\DA^1_\kk}^{\upmu_n}), \star)$ with respect to the
multiplicative set $\{1, \DL_{(\DA^1_\kk,0)}, \DL_{(\DA^1_\kk,0)}
\star \DL_{(\DA^1_\kk,0)}, \dots\}$ by
$(\tilde{\mathcal{M}}_{\DA^1_\kk}^{\upmu_n}, \star)$. Then there is a canonical isomorphism
\begin{equation*}
  \mathcal{M}_\kk \otimes_{K_0(\Var_\kk)}
  (K_0(\Var_{\DA^1_\kk}^{\upmu_n}), \star) 
  \sira
  (\tilde{\mathcal{M}}_{\DA^1_\kk}^{\upmu_n}, \star)
\end{equation*}
of $\mathcal{M}_\kk$-algebras given by $\tfrac{r}{(\DL_\kk)^n}
\otimes a \mapsto \tfrac{ra}{(\DL_{(\DA^1_\kk,0)})^n}$. 
If we compare with 
the isomorphism~\eqref{eq:15} we see that
$\tfrac{b}{(\DL_{\DA^1_\kk})^n} \mapsto
\tfrac{b}{(\DL_{(\DA^1_\kk,0)})^n}$
defines an isomorphism 
\begin{equation}
  \label{eq:71}
  \mathcal{M}_{\DA^1_\kk}^{\upmu_n}
  \sira
  \tilde{\mathcal{M}}_{\DA^1_\kk}^{\upmu_n}
\end{equation}
of $\mathcal{M}_\kk$-modules.

Similarly as above (cf.\ the reasoning around \eqref{eq:22}),
the various $K_0(\Var_\kk)$-algebras
$(K_0(\Var_{\DA^1_\kk}^{\upmu_n}), \star)$ for $n \geq 1$ are
compatible. Hence we obtain the 
$K_0(\Var_\kk)$-algebras
$(K_0(\Var_{\DA^1_\kk}^{\hat{\upmu}}), \star)$
and
\begin{equation*}
  \mathcal{M}_\kk \otimes_{K_0(\Var_\kk)}
  (K_0(\Var_{\DA^1_\kk}^{\hat{\upmu}}), \star) 
  \sira
  (\tilde{\mathcal{M}}_{\DA^1_\kk}^{\hat{\upmu}}, \star)
\end{equation*}
and an isomorphism
\begin{equation}
  \label{eq:62}
  \mathcal{M}_{\DA^1_\kk}^{\hat{\upmu}}
  \sira
  \tilde{\mathcal{M}}_{\DA^1_\kk}^{\hat{\upmu}}
\end{equation}
of $\mathcal{M}_\kk$-modules.

\begin{lemma}
  \label{l:convolution-products-compatible}
  Let $\epsilon \colon  \DA^1_\kk \ra \Spec \kk$ be the structure
  morphism. 
  Then mapping an object $(A \xra{\alpha} \DA^1_\kk) \in
  \Var_{\DA^1_\kk}^{\upmu_n}$ to $(A \xra{\epsilon \alpha} \Spec \kk)
  \in \Var^{\upmu_n}_\kk$ induces a morphism
  \begin{equation}
    \label{epsilon!-star-*}
    \epsilon_! \colon  (\tilde{\mathcal{M}}_{\DA^1_\kk}^{\hat{\upmu}}, \star)
    \ra (\mathcal{M}_k^{\hat{\upmu}}, *)
  \end{equation}
  of $\mathcal{M}_\kk$-algebras.
\end{lemma}

\begin{proof}
  Certainly we have a morphism
  \begin{equation}
    \label{eq:33}
    \epsilon_! \colon  K_0(\Var_{\DA^1_\kk}^{\upmu_n})
    \ra K_0(\Var_k^{\upmu_n})  
  \end{equation}
  of $K_0(\Var_\kk)$-modules.
  If $Z$ is a $\kk$-variety we denote its structure morphism
  $Z \ra \Spec \kk$ by $\epsilon^Z$.
  Let $A, B \in  K_0(\Var_{\DA^1_\kk}^{\upmu_n})$.
  Since $\epsilon_!(A)$ and $\epsilon_!(B)$ are in
  $K_0(\Var_k^{\upmu_n})$,  
  $\epsilon_!(A) * \epsilon_!(B)$ is defined using the fiber
  product over $\kk$.
  Using Remark~\ref{rem:Psi-compatible-with-pushforward-pullback}
  we obtain
  \begin{multline*}
     \epsilon_!(A \star B) 
     = \epsilon_!(\add_!(\Psi(A \times B)))
     = \epsilon^{\DA^1_\kk \times \DA^1_\kk}_!(\Psi(A \times B))
     = \Psi(\epsilon^{\DA^1_\kk \times \DA^1_\kk}_!(A \times B))\\
     = \Psi(\epsilon_!(A) \times \epsilon_!(B))
     = \epsilon_!(A) * \epsilon_!(B).
  \end{multline*}
  Clearly, \eqref{eq:33} maps $[\Spec \kk \xra{0} \DA^1_\kk]$ to
  $[\Spec \kk \ra \Spec \kk]$. Therefore it is a morphism of
  $K_0(\Var_\kk)$-algebras
  $\epsilon_! \colon  (K_0(\Var_{\DA^1_\kk}^{\upmu_n}), \star)
  \ra (K_0(\Var_k^{\upmu_n}),*)$.  
  We can pass to $\hat{\upmu}$. Then
  base change along $K_0(\Var_\kk) \ra
  \mathcal{M}_\kk$ (or noting that
  $\DL_{(\DA^1_\kk,0)}$ goes to $\DL_{\Spec \kk}$) yields a
  morphism
  $\epsilon_! \colon (\tilde{\mathcal{M}}_{\DA^1_\kk}^{\upmu_n},
  \star) \ra (\mathcal{M}_k^{\upmu_n}, *)$
  of $\mathcal{M}_\kk$-algebras.  The lemma follows.
\end{proof}

\section{Motivic vanishing cycles}
\label{sec:motiv-vanish-fiber}

Let $X$ be a smooth connected (nonempty) $\kk$-variety and let 
$V \colon  X \ra \DA^1_\kk$ be a morphism. 
Given $a \in \kk = \DA^1_\kk(\kk)$ we
denote by $X_a$ the scheme theoretic fiber of $V$ over $a$.

We quickly review the definition of the motivic vanishing
cycles. For details we refer to  
\cite[Sect.~3]{guibert-loeser-merle-convolution};
note however that we use slightly different signs, see Remark 
\ref{rem:motivic-fibers-signs} below. 
Following Denef and Loeser, the \define{motivic zeta function of
  $V$ at $a$} is a certain power series 
\begin{equation*}
  Z_{V,a}(T) \in \mathcal{M}^{\hat{\upmu}}_{|X_a|}[[T]]
\end{equation*}
whose coefficients are defined using arc spaces, see
\cite[(3.2.2)]{guibert-loeser-merle-convolution}.
It is possible to evaluate $Z_{V,a}$ at $T=\infty$. 
This is clear if $V$ is constant because then $Z_{V,a}=0$. 
If $V$ is not 
equal to $a$
there is a formula expressing $Z_{V,a}$ in terms of an embedded
resolution of $|X_a| \subset X$
which makes it evident that the
evaluation at $T=\infty$ exists.
 
The \define{motivic nearby fiber
  $\psi_{V,a}$ of $V$ at $a$} is defined to be
the negative of this value at 
infinity, i.\,e.\
\begin{equation*}
  \psi_{V,a} := -Z_{V,a}(\infty) \in
  \mathcal{M}^{\hat{\upmu}}_{|X_a|}. 
\end{equation*}
See \eqref{eq:formula-mot-nearby-fiber} below for a formula for
$\psi_{V,a}$ in terms of an embedded resolution.
The \define{motivic vanishing cycles of $V$ at $a$} are defined by
\begin{equation}
  \label{eq:defi-mot-van-fiber}
  \phi_{V,a}:= [|X_a| \xra{\id} |X_a|] 
  -\psi_{V,a}
  \in
  \mathcal{M}^{\hat{\upmu}}_{|X_a|}.  
\end{equation}
Here $|X_a|$ is endowed with the trivial
$\hat{\upmu}$-action.

\begin{remark}
  \label{rem:constant-potential}
  If $V$ is constant we have $\psi_{V,a}=0$. If $V$ is
  constant 
  $\not=a$ we have $\phi_{V,a}=0$. If $V$ is constant $=a$ we
  have $X=|X_a|$ and $\phi_{V,a}=[X \xra{\id} X]$.
\end{remark}

\begin{remark}
  \label{rem:motivic-fibers-signs} 
  Denef and Loeser choose different signs in the definition of
  the motivic vanishing cycles. In 
  \cite{guibert-loeser-merle-convolution}, the motivic nearby
  fiber (resp.\ motivic vanishing cycles) 
  of $V$ at $0$ is denoted $\mathcal{S}_V$ (resp.\
  $\mathcal{S}_V^\phi$). They are related to our definitions by
  \begin{align*}
    \psi_{V,a} & = \mathcal{S}_{V-a},\\
    \phi_{V,a} & = (-1)^{\dim X} \mathcal{S}_{V-a}^\phi.
  \end{align*}
  Our sign choice for the motivic vanishing cycles is justified in 
  Remark~\ref{rem:justify-sign-vanishing}.
\end{remark}

Let $\Sing(V) \subset X$ be the closed subscheme defined by the 
vanishing of the section $dV \in \Gamma(X, \Omega^1_{X/\kk})$
of the cotangent bundle.
The closed points of $\Sing(V)$ are the critical
points of $V$. Let $\Crit(V) =V(\Sing(V)(\kk)) \subset
\DA^1(\kk)=\kk$ 
be the set of critical values of $W$; it is finite by generic
smoothness on the target. Trivially we have $\Sing(V) \cap
X_a=\emptyset$ if $a$ is not a critical value.

If $Z$ is a scheme locally of finite type over $\kk$ we denote
its open
subscheme
consisting of
regular points by $Z^\reg$. The closed subset $Z^\sing \subset Z$ of
singular points has a unique structure of a reduced closed
subscheme of $Z,$ denoted by $|Z^\sing|$.

\begin{remark}
  \label{rem:SingV-cap-Xa-versus-Xa-sing}
  If $V=a$ then $\Sing(V) \cap X_a=X$ and
  $(X_a)^\sing=\emptyset$. 
  Otherwise
  the singular
  points of $X_a$ are precisely the elements of the
  scheme-theoretic intersection $\Sing(V) \cap X_a,$ i.\,e.\ we
  have the equality
  \begin{equation}
    \label{eq:reduced-SingV-cap-Xa-equals-reduced-Xa-sing}
    |\Sing(V) \cap X_a| = |(X_a)^\sing|  
  \end{equation}
  of $\kk$-varieties. This is
  trivial if $V$ is constant $\not=a,$ and otherwise it follows
  by considering Jacobian matrices.
\end{remark}

Let us prove that the motivic vanishing cycles $\phi_{V,a}$ live 
over $|\Sing(V) \cap X_a|$.

\begin{proposition}
  \label{p:motivic-vanishing-cycles-over-singular-part}
  We have $\phi_{V,a} \in
  \mathcal{M}^{\hat{\upmu}}_{|\Sing(V) \cap X_a|}$ canonically.
\end{proposition}

Therefore we will often view the motivic vanishing cycles
$\phi_{V,a}$ as an 
element of $\mathcal{M}^{\hat{\upmu}}_{|(\Sing(V) \cap X_a|}$ in
the following. 

\begin{proof}
  If $V$ is constant this follows directly from
  Remarks~\ref{rem:constant-potential}
  and \ref{rem:SingV-cap-Xa-versus-Xa-sing}
 
  So let us assume that $V$ is not constant.  
  As in \cite[3.3]{denef-loeser-arc-spaces}
  let
  $h \colon R \ra X$ be an embedded resolution of $|X_a|$
  so that the ideal sheaf of $h^{-1}(|X_a|)$ is the ideal sheaf of
  a simple normal crossing divisor
  (cf.\ \cite[Thm.~3.26]{kollar-singularities} or
  \cite[Thm.~2.2]{villamayor-on-constructive-desingularization} 
  for existence).
  Let $E=h^{-1}(X_a)$ be the divisor
  on $R$ defined by $V \circ h$.
  Let
  $(E_i)_{i \in \Irr(|E|)}$ 
  denote the irreducible components of $|E|$.
  Then $E=\sum_{i \in \Irr(|E|)} m_i E_i$ for unique
  $m_i \in \DN_+$. Let $I \subset \Irr(|E|)$ be given.  Define
  $E_I:= \cap_{i \in I} E_i$ and
  $E_I^\circ:= E_I \setminus \bigcup_{j \in \Irr(|E|) \setminus I}
  E_j$.
  Let $m_I$ be the greatest common divisor of the $m_i$ for
  $i \in I$. Then Denef and Loeser define an unramified Galois
  cover $\tildew{E}^\circ_I \ra E^\circ_I$ with Galois group
  $\upmu_{m_I}$. They establish the formula
  \begin{equation}
    \label{eq:formula-mot-nearby-fiber}
    \psi_{V,a}= \mathcal{S}_{V-a} = \sum_{\emptyset \not=I \subset \Irr(|E|)}
    (1-\DL)^{|I|-1}[\tildew{E}^\circ_I \ra |X_a|],
  \end{equation}
  see \cite[Sect. 3.3 and Def.~3.5.3]{denef-loeser-arc-spaces}.

  Note that $h$ induces an
  isomorphism $h^{-1}(U) \sira U$ where
  $U:= X - |X_a|^\sing$, by 
  part (ii) of
  \cite[Thm.~2.2]{villamayor-on-constructive-desingularization}.
  We can also deduce this from principalization 
  \cite[Thm.~3.26]{kollar-singularities} as follows.
  Since $|X_a|$ has codimension one and $h$ is a
  composition of blow-ups in smooth centers of codimension two
  and higher,  
  $h$ is an isomorphism over an open neighborhood of some regular
  point 
  of $|X_a|$ if $|X_a|$ is non-empty. 
  Since principalization 
  is 
  functorial with respect to \'etale morphisms, $h$ is an
  isomorphism over all regular points of $|X_a|$. 

  We obviously have open embeddings $(X_a)^\reg = |(X_a)^\reg| \subset
  |X_a|^\reg \subset |X_a|$ and hence a closed embedding
  $||X_a|^\sing| \subset |(X_a)^\sing|$.  
  Let $U':= X - (X_a)^\sing \subset U$, so $h^{-1}(U') \sira U'$
  is an isomorphism. Over $X_a \cap U' =(X_a)^\reg$ we obtain the
  isomorphism 
  \begin{equation}
    \label{eq:36}
    h \colon E \cap h^{-1}(U') \sira (X_a)^\reg,
  \end{equation}
  so $E \cap h^{-1}(U')$ is regular.

  If $|I|\geq 2$, then every element $e \in E_I$ lies in $|E|^\sing
  \subset E^\sing$, so $e \not\in E \cap h^{-1}(U')$ and hence
  $h(e) \in (X_a)^\sing$. Therefore  
  $\tildew{E}^\circ_I \ra |X_a|$ factors as 
  $\tildew{E}^\circ_I \ra |(X_a)^\sing| \subset |X_a|$.

  If $r \colon (X_a)^\reg = |(X_a)^\reg| \ra |X_a|$ is the
  inclusion we hence obtain
  \begin{equation*}
    r^*(\psi_{V,a}) = \sum_{i \in \Irr(|E|)}
    [\tildew{E}^\circ_i|_{E_i^\circ \cap h^{-1}(U')} \ra (X_a)^\reg].
  \end{equation*}
  If $E_i^\circ \cap h^{-1}(U')$ is nonempty then $m_i=1$ because
  $E \cap h^{-1}(U')$ is reduced, so $\tildew{E}^\circ_i \ra
  E^\circ_i$ is an isomorphism. Moreover, $E \cap h^{-1}(U')$ is
  the disjoint union of the 
  $E_i^\circ \cap h^{-1}(U')$, for $i \in \Irr(|E|)$.
  These facts and the isomorphism \eqref{eq:36} imply that
  \begin{equation*}
    r^*(\psi_{V,a}) 
    = 
    [(X_a)^\reg \xra{\id} (X_a)^\reg].
  \end{equation*}
  Hence $r^*(\phi_{V,a})=0$ by definition
  \eqref{eq:defi-mot-van-fiber}.
  The decomposition
  $(X_a)^\reg \subset X_a \supset |(X_a)^\sing|$ of $X_a$ into an
  open and a closed reduced subscheme gives rise to a similar
  decomposition
  $(X_a)^\reg=|(X_a)^\reg| \subset |X_a| \supset |(X_a)^\sing|$
  of $|X_a|$. 
  Hence Lemma~\ref{l:open-closed-decomposition} shows that
  $\phi_{V,a} \in \mathcal{M}_{|(X_a)^\sing|}^{\hat{\upmu}}$.
  Since $V$ is not constant, 
  \eqref{eq:reduced-SingV-cap-Xa-equals-reduced-Xa-sing} holds
  true. 
\end{proof}

\begin{corollary}
  \label{c:mot-van-fiber-zero-if-Xa-smooth}
  If $V$ is not constant and $X_a$ is smooth
  then $\phi_{V,a}=0$. 
\end{corollary}

\begin{proof}
  In this case we have 
  $|\Sing(V) \cap X_a| = |(X_a)^\sing|= \emptyset$
  by \eqref{eq:reduced-SingV-cap-Xa-equals-reduced-Xa-sing}.
  More directly, we can take $h=\id$ as an embedded resolution
  in the above proof and obtain 
  $\psi_{V,a}= [|X_a| \xra{\id} |X_a|]$ from
  \eqref{eq:formula-mot-nearby-fiber} and
  hence $\phi_{V,a}=0$. 
\end{proof}

\begin{corollary}
  \label{c:mot-van-fiber-compactification}
  Assume that $X$ is a dense open subset of a smooth
  $\kk$-variety $X'$ and that $V \colon X \ra \DA^1_\kk$ extends
  to a morphism $V' \colon X' \ra \DA^1_\kk$ such that all
  critical points of $V'$ are contained in $X$, i.\,e.\
  $\Sing(V')=\Sing(V') \cap X=\Sing(V)$. 
  Then $\phi_{V,a}=\phi_{V',a}$.
\end{corollary}

\begin{proof}
  If $V$ is constant then $V'$ is constant and
  $X=\Sing(V)=\Sing(V')=X'$, so the claim is trivial.

  Assume that $V$ is not constant. Then we can assume that the 
  embedded resolution $h \colon R \ra X$ of $|X_a|$ from the
  proof of  
  Proposition~\ref{p:motivic-vanishing-cycles-over-singular-part}
  is the restriction to $X$ of a similar embedded resolution 
  $h' \colon R' \ra X'$ 
  of $|X'_a|$.
  Let $s \colon |(X_a)^\sing| \ra |X_a|$ 
  and $s' \colon |(X'_a)^\sing| \ra |X'_a|$ denote the closed
  embeddings. 
  Then $\phi_{V,a}=s_! s^*(\phi_{V,a})$ 
  by (the proof of)
  Proposition~\ref{p:motivic-vanishing-cycles-over-singular-part}
  and
  \begin{equation*}
    s^*\phi_{V,a} = 
    [|(X_a)^\sing| \ra |(X_a)^\sing|]
    -s^*(\psi_{V,a})
  \end{equation*}
  by definition 
  \eqref{eq:defi-mot-van-fiber}.
  Similarly, we have $\phi_{V',a}=s'_! s'^*(\phi_{V',a})$ and
  \begin{equation*}
    s'^*\phi_{V',a}= 
    [|(X'_a)^\sing| \ra |(X'_a)^\sing|]
    -s'^*(\psi_{V',a}).
  \end{equation*}
  By assumption and Remark~\ref{rem:SingV-cap-Xa-versus-Xa-sing}
  we have $|(X_a)^\sing|=|\Sing(V) \cap X_a|=|\Sing(V') \cap
  X'_a|=|(X'_a)^\sing|$. Therefore it is enough to show that
  $s^*(\psi_{V,a})=s'^*(\psi_{V',a})$. But both
  expressions have 
  explicit formulas obtained from
  equation \eqref{eq:formula-mot-nearby-fiber}; these expressions
  coincide since the Galois coverings 
  $\tildew{E}^\circ_I \ra E^\circ_I$ and 
  $\tildew{E}'^\circ_{I'} \ra E'^\circ_{I'}$ 
  constructed for 
  $h \colon R \ra X$ 
  and $h' \colon R' \ra X'$ are compatible and give rise to
  isomorphic varieties 
  with $\hat{\upmu}$-action
  over  
  $|(X_a)^\sing|=|(X'_a)^\sing|$.
\end{proof}

\section{Motivic Thom-Sebastiani theorem}
\label{sec:thom-sebastiani}

Let $V \colon  X \ra \DA^1_\kk$ be as in the previous
section~\ref{sec:motiv-vanish-fiber}.  
Let $Y$ be a smooth connected (nonempty) $\kk$-variety with a 
morphism $W \colon Y \ra \DA^1_\kk$. 
We define $V \circledast W$ to be the composition
\begin{equation*}
  V \circledast W \colon X \times Y \xra{V \times W} \DA^1_\kk \times
  \DA^1_\kk \xra{+} 
  \DA^1_\kk.  
\end{equation*}

\begin{theorem}
  [{Motivic Thom-Sebastiani Theorem,
  \cite[Thm.~5.18]{guibert-loeser-merle-convolution}}] 
  \label{t:thom-sebastiani-GLM-motivic-vanishing-cycles}
  Consider morphisms $V \colon X \ra \DA^1_\kk$ and $W \colon Y \ra
  \DA^1_\kk$ 
  as above and let $a, b \in \kk$. Let $i_{a,b}$ be the
  inclusion $|X_a| \times 
  |Y_b| \ra |(X \times Y)_{a+b}|$. 
  Then
  \begin{equation}
    \label{eq:thom-sebastiani-GLM-motivic-vanishing-cycles}
    i_{a,b}^*(\phi_{V \circledast W,a+b})
    =
    \Psi(\phi_{V,a} \times \phi_{W,b}) 
  \end{equation}
  in $\mathcal{M}^{\hat{\upmu}}_{|X_a| \times |Y_b|}$ where 
  $\phi_{V,a} \times \phi_{W,b}$ is the obvious element of
  $\mathcal{M}^{\hat{\upmu}
  \times \hat{\upmu}}_{|X_a| \times |Y_b|}$.
\end{theorem}

\begin{proof}
  Using Remark~\ref{rem:motivic-fibers-signs}
  this is precisely
  \cite[Thm.~5.18]{guibert-loeser-merle-convolution}.  

  We would like to emphasize that
  \eqref{eq:thom-sebastiani-GLM-motivic-vanishing-cycles} also
  holds if $V$ or $W$ is constant. 
  Assume first that both $V$ and $W$ are constant; if $V=a$ and
  $W=b$ then both sides of
  \eqref{eq:thom-sebastiani-GLM-motivic-vanishing-cycles} are
  equal to $[X \times Y \xra{\id} X \times Y]$ (use
  Remark~\ref{rem:constant-potential} and 
  Example~\ref{ex:p-product-over-k-trivial-on-one-factor}); otherwise
  both sides are zero.

  Now assume that $V$ is not constant but $W$ is. If $W
  \not= b$ then again both sides of 
  \eqref{eq:thom-sebastiani-GLM-motivic-vanishing-cycles} are
  zero. If $W=b$ choose an embedded resolution of $|X_a| \subset
  X$ as in the
  proof of 
  Proposition~\ref{p:motivic-vanishing-cycles-over-singular-part}
  and obtain an explicit expression for $\phi_{V,a}$. If we take
  the product of 
  this embedded resolution with $Y=Y_b$ we obtain an embedded
  resolution of $|(X \times Y)_{a+b}|=|X_a| \times Y \subset X
  \times Y$ and an explicit expression for $\phi_{V \circledast W,a+b}$.
  Now use again Remark~\ref{rem:constant-potential} and 
  Example~\ref{ex:p-product-over-k-trivial-on-one-factor}.
\end{proof}

We want to globalize this theorem.
Since the set $\Crit(V) \subset \kk$ of
critical values of $V$ is finite, we have
\begin{equation}
  \label{eq:63}
  |\Sing(V)| = \coprod_{a \in \Crit(V)} |\Sing(V) \cap X_a|.
\end{equation}
Proposition~\ref{p:motivic-vanishing-cycles-over-singular-part}
shows that we can view 
$\phi_{V,a}$ as an element of
$\mathcal{M}^{\hat{\upmu}}_{|\Sing(V)|}$. Define
\begin{equation*}
  \tildew{\phi}_V:= \sum_{a \in \Crit(V)} \phi_{V,a}
  \in \mathcal{M}^{\hat{\upmu}}_{|\Sing(V)|}. 
\end{equation*}
Of course we could have equivalently taken the sum over all $a
\in \kk$,
by Corollary~\ref{c:mot-van-fiber-zero-if-Xa-smooth}
and Remark~\ref{rem:constant-potential}.

We obviously have
\begin{equation}
  \label{eq:Sing-V*W}
  \Sing(V \circledast W) =   \Sing(V) \times \Sing(W)
\end{equation}
and hence
$\Crit(W*V)=\Crit(W)+\Crit(V):=\{a+b\mid a \in
\Crit(W), b \in \Crit(V)\}$. 

\begin{corollary}
  [{Global motivic Thom-Sebastiani}]
  \label{c:global-thom-sebastiani}
  Let $V \colon X \ra \DA^1_\kk$ and $W \colon Y \ra \DA^1_\kk$
  be as above. Then
  \begin{equation}
    \label{eq:global-thom-sebastiani}
    \tildew{\phi}_{V \circledast W}
    =
    \Psi(\tildew{\phi}_V \times \tildew{\phi}_W) 
  \end{equation}
  in $\mathcal{M}^{\hat{\upmu}}_{|\Sing(V)| \times |\Sing(W)|}$ where 
  $\tildew{\phi}_V \times \tildew{\phi}_W$ is the obvious element of
  $\mathcal{M}^{\hat{\upmu}
  \times \hat{\upmu}}_{|\Sing(V)| \times |\Sing(W)|}$.
\end{corollary}

\begin{proof}
  Let $s_a \colon |\Sing(V) \cap X_a| \ra |\Sing(V)|$ be
  the closed embedding. Then obviously
  $s_a^*(\tildew{\phi})=\phi_{V,a}$. 
  Define $s'_b$ and $s''_c$ similarly for $W$ and $V \circledast W$.
  From \eqref{eq:63} and \eqref{eq:Sing-V*W} we see that
  $|\Sing(V \circledast W)|$ is the disjoint finite union of its
  closed  
  subvarieties $|\Sing(V) \cap X_a| \times |\Sing(W) \cap Y_b|$ where
  $(a,b) \in \Crit(V) \times \Crit(W)$.
  By Lemma~\ref{l:open-closed-decomposition} it is therefore
  enough to show that both sides of
  \eqref{eq:global-thom-sebastiani}
  coincide when restricted to each of these subvarieties.
  Consider the following commutative diagram
  \begin{equation*}
    \xymatrix{
      {|\Sing(V) \cap X_a| \times |\Sing(W) \cap Y_b|}
      \ar[d]^-{s_a \times s'_b} 
      \ar[r]^-{\iota_{a,b}}
      &
      {|\Sing(V \circledast W) \cap (X \times
        Y)_{a+b}|}
      \ar[d]^-{s''_{a+b}}\\
      {|\Sing(V)| \times |\Sing(W)|}
      \ar@{}[r]|-{=}
      & 
      {|\Sing(V \circledast W)|.}
    }
  \end{equation*}
  If we apply $(s_a \times s'_b)^*$ to both sides of
  \eqref{eq:global-thom-sebastiani}
  we obtain on the left
  \begin{equation*}
    \iota_{a,b}^*((s''_{a+b})^*(\tildew{\phi}_{V \circledast W}))
    = \iota_{a,b}^*(\phi_{V \circledast W,a+b})
  \end{equation*}
  and on the right
  \begin{equation*}
    \Psi((s_a \times s'_b)^*(\tildew{\phi}_V \times
    \tildew{\phi}_W)) 
    =
    \Psi(s_a^*(\tildew{\phi}_V) \times (s'_b)^*(\tildew{\phi}_W)) 
    =
    \Psi(\phi_{V,a} \times \phi_{W,b})
  \end{equation*}
  where we use
  Remark~\ref{rem:Psi-compatible-with-pushforward-pullback}. 
  But 
  $\iota_{a,b}^*(\phi_{V \circledast W,a+b})
  = \Psi(\phi_{V,a} \times \phi_{W,b})$
  is just a reformulation of
  Theorem~\ref{t:thom-sebastiani-GLM-motivic-vanishing-cycles}
  using 
  Proposition~\ref{p:motivic-vanishing-cycles-over-singular-part}
  and Remark~\ref{rem:Psi-compatible-with-pushforward-pullback}. 
\end{proof}


\begin{definition}
  \label{d:total-motivic-vanishing-cycles}
  For $V \colon  X \ra \DA^1_\kk$ as above we define 
  \begin{equation}
    \label{eq:phi-V-DA1}
    (\phi_V)_{\DA^1_\kk} := V_!(\tildew{\phi}_V)=\sum_{a \in \kk}
    V_!(\phi_{V,a}) 
    \in \tilde{\mathcal{M}}^{\hat{\upmu}}_{\DA^1_\kk} 
  \end{equation}
  where we use the isomorphism \eqref{eq:62}
  in order to change the target of
  $V_! \colon \mathcal{M}_{|\Sing(V)|}^{\hat{\upmu}} \ra 
  \mathcal{M}^{\hat{\upmu}}_{\DA^1_\kk}$ to
  $\tilde{\mathcal{M}}^{\hat{\upmu}}_{\DA^1_\kk}$.
\end{definition}

Recall the convolution product
\eqref{eq:convolution-A1}. 

\begin{corollary}
  \label{c:thom-sebastiani-motivic-vanishing-cycles}
  Let $V \colon X \ra \DA^1_\kk$ and $W \colon Y \ra \DA^1_\kk$
  be as above. 
  Then
  \begin{equation*}
    (\phi_{V \circledast W})_{\DA^1_\kk}
    =
    (\phi_{V})_{\DA^1_\kk}
    \star
    (\phi_{W})_{\DA^1_\kk}
  \end{equation*}
  in
  $(\tilde{\mathcal{M}}^{\hat{\upmu}}_{\DA^1_\kk},\star)$.
\end{corollary}

\begin{proof}
  Just apply $(V \circledast W)_!=\add_!(V \times W)_!$ to 
  \eqref{eq:global-thom-sebastiani} and note that
  \begin{multline*}
    \add_!((V \times W)_!(\Psi(\tildew{\phi}_V \times
    \tildew{\phi}_W)))
    =\add_!(\Psi((V \times W)_!(\tildew{\phi}_V \times
    \tildew{\phi}_W)))\\
    =\add_!(\Psi(V_!(\tildew{\phi}_V) \times
    W_!(\tildew{\phi}_W)))
    =
    (\phi_{V})_{\DA^1_\kk}
    \star
    (\phi_{W})_{\DA^1_\kk}.
  \end{multline*}
  using
  Remark~\ref{rem:Psi-compatible-with-pushforward-pullback}. 
\end{proof}


Define
$\ul{\phi_V}$ as the image of
$(\phi_V)_{\DA^1_\kk}$ under the morphism $\epsilon_!$ 
of rings from Lemma~\ref{l:convolution-products-compatible}, i.\,e.\
\begin{equation}
  \label{eq:phi-V-k}
  \ul{\phi_V} := 
  \epsilon_!((\phi_V)_{\DA^1_\kk})
  =
  \sum_{a \in \kk} (\epsilon_a)_!(\phi_{V,a})
  \in \mathcal{M}^{\hat{\upmu}}_{\kk} 
\end{equation}
where $\epsilon_a \colon |\Sing(V) \cap X_a| \ra \Spec \kk$
denotes the structure morphism.

\begin{corollary}
  \label{c:thom-sebastiani-motivic-vanishing-cycles-over-k}
  We have
  $\ul{\phi_{V \circledast W}}
  =
  \ul{\phi_{V}}
  *
  \ul{\phi_{W}}$
  in
  $\mathcal{M}^{\hat{\upmu}}_{\kk}$.
\end{corollary}

\begin{proof}
  This is obvious from
  Lemma~\ref{l:convolution-products-compatible} 
  and Corollary~\ref{c:thom-sebastiani-motivic-vanishing-cycles}.
\end{proof}

\section{Motivic vanishing cycles measure}
\label{sec:motiv-vanish-fiber-1}

\subsection{Some reminders}
\label{sec:some-reminders}

We recall some facts from
\cite[3.7-3.9]{guibert-loeser-merle-convolution}.
Let $X$ be a smooth connected $\kk$-variety and $V
\colon X \ra \DA^1_\kk$ a morphism. Let $U \subset X$ be a dense
open subvariety. Then Guibert, Loeser and Merle
define 
in \cite[Prop.~3.8]{guibert-loeser-merle-convolution}
an element 
\begin{equation*}
  \mathcal{S}_{V, U, X} \in
  \mathcal{M}_{|X_0|}^{\hat{\upmu}}  
\end{equation*}
(they denote this element just
by $\mathcal{S}_{V,U}$).

\begin{remark}
  The remarks at the end of 
  \cite[3.8]{guibert-loeser-merle-convolution}
  (and Remarks~\ref{rem:motivic-fibers-signs} and
  \ref{rem:constant-potential})
  imply:
  if $U=X$ we have
  \begin{equation}
    \label{eq:43} 
    \mathcal{S}_{V, X, X}=
    \mathcal{S}_V=\psi_{V,0};
  \end{equation}
  if $V=0$ we have
  \begin{equation*}
     \mathcal{S}_{V, U, X}=0=\psi_{V,0}=\mathcal{S}_{V}
  \end{equation*}
\end{remark}

\begin{theorem}
  [{\cite[Thm.~3.9]{guibert-loeser-merle-convolution}}]  
  \label{t:GLM-measure}
  Let $\alpha \colon A \ra \DA^1_\kk$ be an $\DA^1_\kk$-variety. 
  Then there exists a unique $\mathcal{M}_\kk$-linear map
  \begin{equation*}
    \mathcal{S}_\alpha^{\mathcal{M}_A} \colon \mathcal{M}_A \ra
    \mathcal{M}^{\hat{\upmu}}_{|A_0|} 
  \end{equation*}
  such that for every proper morphism $V' \colon Z \ra A$
  where $Z$ 
  is a smooth and connected $\kk$-variety, and every dense open
  subvariety $U$ of $Z$ 
  we have
  \begin{equation*}
    \mathcal{S}_\alpha^{\mathcal{M}_A}([U \ra A]) =
    V'_!(\mathcal{S}_{\alpha \circ V', U, X}).
  \end{equation*}
\end{theorem}

Note that given any morphism $V \colon U \ra A$ where $U$ is a smooth
connected $\kk$-variety, there is a smooth connected
$\kk$-variety $Z$ containing $U$ as a dense open subscheme and a
proper morphism $V' \colon Z \ra A$ extending $V$ (use Nagata
compactification and resolve the singularities).

In particular, if $V$ is a proper morphism, then by definition
of $\mathcal{S}_\alpha^{\mathcal{M}_A}$ and using
\eqref{eq:43} we have
\begin{equation}
  \label{eq:42}
  \mathcal{S}_\alpha^{\mathcal{M}_A}([U \xra{V} A]) 
  = V_!(\mathcal{S}_{\alpha \circ V, U, U})
  = V_!(\mathcal{S}_{\alpha \circ V})
  = V_!(\psi_{\alpha \circ V,0}).
\end{equation}
In particular, if $A$ is smooth and connected we obtain 
$\mathcal{S}_\alpha^{\mathcal{M}_A}([A \xra{\id} A]) 
= \mathcal{S}_{\alpha}=\psi_{\alpha,0}$ which justifies the notation
$\mathcal{S}_\alpha^{\mathcal{M}_A}$.

We will apply Theorem~\ref{t:GLM-measure} only in the case that
$\alpha$ is a translation $\DA^1_\kk \ra \DA^1_\kk$, $x \mapsto
x-a$, for some $a \in \kk$.

\subsection{Additivity of the motivic vanishing cycles}
\label{sec:addit-vanish-cycl}

Theorem~\ref{t:GLM-measure} has the following consequence.


\begin{theorem}
  \label{t:total-vanishing-cycles-group-homo}
  There exists a unique $\mathcal{M}_\kk$-linear map
  \begin{equation*}
    \Phi' \colon  \mathcal{M}_{\DA^1_\kk} \ra
    \tilde{\mathcal{M}}_{\DA^1_\kk}^{\hat{\upmu}} 
  \end{equation*}
  such that 
  \begin{equation*}
    \Phi'([X \xra{V} \DA^1_\kk]) = (\phi_{V})_{\DA^1_\kk}
  \end{equation*}
  for all proper morphisms $V \colon  X \ra \DA^1_\kk$ of
  $\kk$-varieties 
  where $X$ is smooth over $\kk$ and connected
  (for the definition of $(\phi_{V})_{\DA^1_\kk}$ see
  Definition~\ref{d:total-motivic-vanishing-cycles}). 
\end{theorem}

\begin{proof}
  Uniqueness is clear 
  since $K_0(\Var_{\DA^1_\kk})$ is generated by the classes of
  proper morphisms $V \colon  X \ra \DA^1_\kk$ of $\kk$-varieties
  with $X$ 
  connected and smooth over $\kk$ 
  (and relations given by the blowing-up relations), see
  \cite[Thm.~5.1]{bittner-euler-characteristic}.

  Let $a \in \kk$ and let
  $\gamma_a \colon |X_a| \ra \{a\}=\Spec \kk$ be
  the obvious morphism. Apply Theorem~\ref{t:GLM-measure} to the
  morphism $\alpha \colon \DA^1_\kk \ra \DA^1_\kk$,
  $x \mapsto x-a$. We obtain an $\mathcal{M}_\kk$-linear map
  $\mathcal{M}_{\DA^1_\kk} \ra \mathcal{M}_{\{a\}}^{\hat{\upmu}}$
  that maps $[V \colon X \ra \DA^1_\kk]$ to
  $-(\gamma_a)_!(\psi_{V,a})$
  (use \eqref{eq:42}; we add a global minus sign)
  whenever $V \colon X \ra \DA^1_\kk$
  is proper with $X$ connected and smooth over $\kk$.  

  Obviously there
  is a unique 
  $\mathcal{M}_\kk$-linear map
  $\mathcal{M}_{\DA^1_\kk} \ra
  \mathcal{M}_{\{a\}}^{\hat{\upmu}}$
  mapping $[V \colon  X \ra
  \DA^1_\kk]$ to 
  $[|X_a| \ra \{a\}]=(\gamma_a)_!([|X_a| \ra |X_a|)$
  for any morphism $V \colon  X \ra \DA^1_\kk$ of 
  $\kk$-varieties.
  
  Let $\Phi'_a \colon  \mathcal{M}_{\DA^1_\kk} \ra
  \mathcal{M}_{\{a\}}^{\hat{\upmu}}$ be the sum of these two
  maps. If
  $V \colon X \ra \DA^1_\kk$
  is proper with $X$ connected and smooth over $\kk$ we have 
  $\Phi'_a([V \colon X \ra \DA^1_\kk]) =(\gamma_a)_!(\phi_{V,a})$
  by the  
  definition of the motivic vanishing cycles
  \eqref{eq:defi-mot-van-fiber}. 

  For any $a \in \kk,$ let $i_a \colon \{a\} \ra \DA^1_\kk$ be the
  inclusion. Observe that 
  \begin{equation*}
    \sum_{a \in \kk} (i_a)_!(\Phi'_a) \colon 
    \mathcal{M}_{\DA^1_\kk} \ra \mathcal{M}_{\DA^1_\kk}^{\hat{\upmu}}    
  \end{equation*}
  is
  well defined since for any given $m \in \mathcal{M}_{\DA^1_\kk}$
  only finitely many $\Phi'_a(m)$ are nonzero. 
  The composition of this morphism with the isomorphism \eqref{eq:62}
  has the required properties. This proves existence.
\end{proof}

\begin{remark}
  \label{rem:Phi'-on-unity-and-A1-0-A1}
  If $Z$ is a smooth $\kk$-variety we have $\Phi'([Z \xra{0}
  \DA^1_\kk])= [Z \xra{0} \DA^1_\kk]$. 
  If $Z$ is proper over $\kk$ this follows from 
  Remark~\ref{rem:constant-potential}. Otherwise we can
  compactify $Z$ to a smooth proper $\kk$-variety $\ol{Z}$ such
  that $\ol{Z}-Z$ is a simple normal crossing divisor and then
  express the class of $Z$ in terms of $\ol{Z}$ and the various
  smooth intersections of the involved smooth prime divisors.
  
  In particular we have
  $\Phi'([\Spec \kk \xra{0} \DA^1_\kk])=[\Spec \kk \xra{0}
  \DA^1_\kk]$
  and 
  $\Phi'([\DA^1_\kk \xra{0} \DA^1_\kk])=[\DA^1_\kk \xra{0}
  \DA^1_\kk]$
\end{remark}

\begin{remark}
  \label{rem:justify-sign-vanishing}
  We 
  keep our promise from
  Remark~\ref{rem:motivic-fibers-signs}
  to justify our sign choice.
  We do this
  by showing that
  Theorem~\ref{t:total-vanishing-cycles-group-homo} 
  does not hold if the right hand side of
  \eqref{eq:phi-V-DA1} 
  is replaced by
  $\sum_{a \in \kk} V_!(\mathcal{S}^\phi_{V-a})$.
  Assume that there is morphism $\Xi \colon
  \mathcal{M}_{\DA^1_\kk} \ra
  \mathcal{M}_{\DA^1_\kk}^{\hat{\upmu}}
  \cong
  \tilde{\mathcal{M}}_{\DA^1_\kk}^{\hat{\upmu}}$
  of abelian groups
  such that 
  $\Xi([X \xra{V} \DA^1_\kk]) = 
  \sum_{a \in \kk} V_!(\mathcal{S}^\phi_{V-a})$
  for all proper morphisms $V \colon  X \ra \DA^1_\kk$ of
  $\kk$-varieties 
  where $X$ is smooth over $\kk$ and connected.
  Remark~\ref{rem:constant-potential} implies that
  $\Xi([Z \xra{0}
  \DA^1_\kk])= (-1)^{\dim Z}[Z \xra{0} \DA^1_\kk]$ for all smooth proper
  connected $\kk$-varieties $Z$.

  Let $X$ be a smooth proper connected 2-dimensional
  $\kk$-variety 
  and $\tildew{X}$ its blowup in a 
  closed point $Y =\{x\} \subset X$. Let $E$ be the exceptional
  divisor. We view $X, \tildew{X}, Y, E$ as $\DA^1_\kk$-varieties
  via the zero morphism to $\DA^1_\kk$.
  In $K_0(\Var_{\DA^1_\kk})$ we obviously have $[X]-[Y]=
  [\tildew{X}] -[E]$. So if we apply $\Xi$ we obtain
  $[X]-[Y]=[\tildew{X}] +[E]$ since $E$ has odd dimension. 
  We obtain $2[E]=0$ in 
  $\mathcal{M}_{\DA^1_\kk}^{\hat{\upmu}}$.
  Let us explain why this is a contradiction.
  Note that $E \cong
  \DP^1_\kk$. Pulling back via the inclusion $\Spec \kk \xra{0}
  \DA^1_\kk$ and forgetting the group action shows that
  $2[\DP^1_\kk]=0$ in $\mathcal{M}_\kk$. 
  Taking the topological Euler characteristic with compact
  support (see \cite[Example
  4.3]{nicaise-sebag-grothendieck-ring-of-varieties}) yields the
  contradiction $4=0$ in $\DZ$.
\end{remark}

\begin{remark}
  \label{rem:Phi'-on-smooth-proper-morphisms}
  If $V \colon X \ra \DA^1_\kk$ is a smooth and proper morphism
  then 
  $\Phi'([X \xra{V} \DA^1_\kk])=0$. This follows from
  Corollary~\ref{c:mot-van-fiber-zero-if-Xa-smooth};
  note that
  $X$ is smooth over $\kk$ and $V$ is not constant (if $X$ is
  nonempty). 
\end{remark}

\begin{remark}
  \label{rem:Phi'-of-L-A1}
  We claim that $\Phi'(\DL_{\DA^1_\kk})=0$.
  Indeed, we have
  \begin{equation*}
    \DL_{\DA^1_\kk}=
    [\DA^1_{\DA^1_\kk}]=[\DA^1_\kk \times \DA^1_\kk \ra \DA^1_\kk]
    =
    [\DA^1_\kk \times \DP^1_\kk \ra \DA^1_\kk]
    -
    [\DA^1_\kk \times \Spec \kk \ra \DA^1_\kk]
  \end{equation*}
  in $K_0(\Var_{\DA^1_\kk})$. Now apply 
  Remark~\ref{rem:Phi'-on-smooth-proper-morphisms}.
  In fact, this argument 
  together with the compactification argument from
  Remark~\ref{rem:Phi'-on-unity-and-A1-0-A1}
  shows: if $Z$ is any smooth
  $\kk$-variety, then $\Phi'$ maps the class of the projection
  $\DA^1_\kk \times Z \ra \DA^1_\kk$ to zero.
\end{remark}

\begin{remark}
  \label{rem:Phi'-cannot-be-ring-morphism}
  Remark~\ref{rem:Phi'-of-L-A1} shows that
  the morphism $\Phi'$ from
  Theorem~\ref{t:total-vanishing-cycles-group-homo}
  is not a morphism of rings if we consider the usual 
  multiplication on
  $\mathcal{M}_{\DA^1_\kk}$: it
  maps the invertible element 
  $\DL_{\DA^1_\kk}$ to zero and hence would be the zero morphism
  (which it is not, by
  Remark~\ref{rem:Phi'-on-unity-and-A1-0-A1}).
  Therefore it seems presently more appropriate to restrict
  $\Phi'$ to  
  $K_0(\Var_{\DA^1_\kk})$. 
  See however Remark~\ref{rem:A1-0-A1-sent-to-invertible} below.
\end{remark}

\subsection{The motivic vanishing cycles measure}
\label{sec:motiv-vanish-cycl}

We define $\Phi$ to be the $K_0(\Var_\kk)$-linear composition
\begin{equation}
  \label{eq:52}
  \Phi \colon K_0(\Var_{\DA^1_\kk}) \ra 
  \mathcal{M}_{\DA^1_\kk} \xra{\Phi'}
  \tilde{\mathcal{M}}_{\DA^1_\kk}^{\hat{\upmu}} 
\end{equation}
where the second map is the morphism $\Phi'$ from
Theorem~\ref{t:total-vanishing-cycles-group-homo}.

Now we can state our main theorem which says that $\Phi$ is a
ring morphism if we equip source 
$K_0(\Var_{\DA^1_\kk})=  K_0(\Var^{\upmu_1}_{\DA^1_\kk})$ and
target $\tilde{\mathcal{M}}_{\DA^1_\kk}^{\hat{\upmu}}$ 
with the convolution product $\star$ 
from
section~\ref{sec:conv-vari-over}, see in particular
Remark~\ref{rem:no-group-action} and Proposition~\ref{p:A1-convolution-comm-ass-unit}.

\begin{theorem}
  \label{t:total-motivic-vanishing-cycles-ring-homo}
  The map \eqref{eq:52}
  from Theorem~\ref{t:total-vanishing-cycles-group-homo}
  is a morphism 
  \begin{equation*}
    \Phi \colon  (K_0(\Var_{\DA^1_\kk}), \star) \ra
    (\tilde{\mathcal{M}}_{\DA^1_\kk}^{\hat{\upmu}},\star) 
  \end{equation*}
  of $K_0(\Var_\kk)$-algebras.
  By composing with
  \eqref{epsilon!-star-*}
  we obtain a morphism
  \begin{equation}
    \label{eq:epsilon-Phi}
    \epsilon_! \circ \Phi \colon  (K_0(\Var_{\DA^1_\kk}), \star) \ra
    (\mathcal{M}_{\kk}^{\hat{\upmu}},*) 
  \end{equation}
  of $K_0(\Var_\kk)$-algebras. We call these two morphisms
  \textbf{motivic vanishing cycles measures}.
\end{theorem}

\begin{proof}
  The second claim is obvious from
  Lemma~\ref{l:convolution-products-compatible}, so let us prove
  the first claim.
  Remark~\ref{rem:Phi'-on-unity-and-A1-0-A1} shows that $\Phi$ maps
  the identity element to the identity element.
  Remark~\ref{rem:Phi'-on-unity-and-A1-0-A1} shows that $\Phi$ is
  compatible with the algebra structure maps, cf.\ \eqref{eq:46}.

  We use that $K_0(\Var_{\DA^1_\kk})$ is generated by the classes of
  projective morphisms $V \colon  X \ra \DA^1_\kk$ with $X$
  a connected quasi-projective $\kk$-variety that is smooth over
  $\kk$ 
  (and relations given by the blowing-up relations), see
  \cite[Thm.~5.1]{bittner-euler-characteristic}.

  So let $X$ and $Y$ be connected quasi-projective $\kk$-varieties that
  are smooth over $\kk$ and let $V \colon X \ra \DA^1_\kk$ and $W \colon Y \ra
  \DA^1_\kk$ be projective morphisms.
  Then we know by
  Theorem~\ref{t:total-vanishing-cycles-group-homo}
  and
  Corollary
  \ref{c:thom-sebastiani-motivic-vanishing-cycles}
  that
  \begin{equation*}
    \Phi([X \xra{V} \DA^1_\kk]) \star
    \Phi([Y \xra{W} \DA^1_\kk])
    =
    (\phi_{V})_{\DA^1_\kk}
    \star
    (\phi_{W})_{\DA^1_\kk}
    = (\phi_{V \circledast W})_{\DA^1_\kk}.
  \end{equation*}
  Our aim is to show that
  \begin{equation*}
    (\phi_{V \circledast W})_{\DA^1_\kk}=  
    \Phi([X \times Y \xra{V \circledast W} \DA^1_\kk]).
  \end{equation*}
  This is not obvious since $V \circledast W \colon  X \times Y \ra
  \DA^1_\kk$ is not 
  proper in general. 

  We apply Proposition~\ref{p:compactification} below and
  use notation from there.
  We obtain the equality
  \begin{multline*}
    [X\times Y \xra{V \circledast W} \DA^1_\kk]
    =
    [Z \xra{h} \DA^1_\kk]
    -\sum_{i}[D_i\xra{h_i} \DA^1_\kk]
    +\sum_{i<j}[D_{ij} \xra{h_{ij}} \DA^1_\kk]\\
    - \dots + (-1)^s[D_{12 \dots s} \xra{h_{12 \dots s}} \DA^1_\kk]
  \end{multline*}
  in $K_0(\Var_{\DA^1_\kk})$.
  On the right-hand side, $Z$ and all $D_{i_1 \dots i_p}$ 
  are smooth quasi-projective $\kk$-varieties, $h$ is a projective
  morphism, and all 
  $h_{i_1 \dots i_p}$ are projective and smooth morphisms, 
  by part
  \ref{enum:compactification-intersections-proper-and-smooth} 
  of
  Proposition~\ref{p:compactification}.
  Hence we can compute $\Phi([X\times Y \xra{V \circledast W}
  \DA^1_\kk]$ now. 
  Remark~\ref{rem:Phi'-on-smooth-proper-morphisms}
  shows that
  $\Phi$ vanishes
  on all
  $[D_{i_1 \dots i_p}\xra{h_{i_1 \dots i_p}} \DA^1_\kk]$. We
  obtain
  \begin{equation*}
    \Phi([X\times Y \xra{V \circledast W} \DA^1_\kk])=\Phi([Z \xra{h}
    \DA^1_\kk]) = (\phi_{h})_{\DA^1_\kk}
  \end{equation*}
  and are left to show that
  \begin{equation*}
    (\phi_{h})_{\DA^1_\kk} = (\phi_{V \circledast W})_{\DA^1_\kk}.
  \end{equation*}
  But this holds true by 
  Corollary~\ref{c:mot-van-fiber-compactification}
  which we can apply by
  parts
  \ref{enum:compactification-diagram},
  \ref{enum:compactification-critical-points},
  \ref{enum:compactification-boundary-snc}
  of Proposition~\ref{p:compactification}.
\end{proof}

\begin{remark}
  \label{rem:evaluate-on-product}
  If $X$ and $Y$ are smooth connected $\kk$-varieties and $V
  \colon X \ra \DA^1_\kk$ and $W \colon Y \ra \DA^1_\kk$ are
  proper morphisms then
  Theorems~\ref{t:total-motivic-vanishing-cycles-ring-homo}
  and \ref{t:thom-sebastiani-GLM-motivic-vanishing-cycles}
  show that
  \begin{equation*}
    \Phi([X \times Y \xra{V \circledast W} \DA^1_\kk]) =
    \Phi([X \xra{V} \DA^1_\kk]) \ast \Phi([Y \xra{W} \DA^1_\kk]) 
    =
    (\phi_{V})_{\DA^1_\kk}
    \ast 
    (\phi_{V})_{\DA^1_\kk}
    =
    (\phi_{V \circledast W})_{\DA^1_\kk}.
  \end{equation*}
  So even though $V \circledast W$ might not be proper, the
  motivic vanishing cycles measure $\Phi$ sends it to 
  $(\phi_{V \circledast W})_{\DA^1_\kk}$. 
\end{remark}

\begin{remark}
  \label{rem:A1-0-A1-sent-to-invertible}
  Recall the element
  $\DL_{(\DA^1_\kk,0)}=[\DA^1_\kk \xra{0} \DA^1_\kk] \in
  K_0(\Var^{\upmu_n}_{\DA^1_\kk})$
  defined in \eqref{eq:def-L-A1-0}. For $n=1$ it is an element of
  $K_0(\Var_{\DA^1_\kk})$ and we have
  $(K_0(\Var_{\DA^1_\kk}), \star)[(\DL_{(\DA^1_\kk,0)})^{-1}]=
  (\tilde{\mathcal{M}}_{\DA^1_\kk}, \star)$.
  Remark~\ref{rem:Phi'-on-unity-and-A1-0-A1}  
  shows that $\Phi(\DL_{(\DA^1_\kk,0)})= \DL_{(\DA^1_\kk,0)}$
  where we view $\DL_{(\DA^1_\kk,0)}$ in the obvious way as an element of
  $K_0(\Var^{\hat{\upmu}}_{\DA^1_\kk})$. 
  Therefore, $\Phi$ factors as the composition
  \begin{equation*}
    \Phi \colon  (K_0(\Var_{\DA^1_\kk}), \star) 
    \ra 
    (\tilde{\mathcal{M}}_{\DA^1_\kk},\star) 
    \ra
    (\tilde{\mathcal{M}}_{\DA^1_\kk}^{\hat{\upmu}},\star). 
  \end{equation*}
  The second map is a morphism of $\mathcal{M}_\kk$-algebras.
  It is, up to the isomorphism
  $\mathcal{M}_{\DA^1_\kk}
  \sira
  \tilde{\mathcal{M}}_{\DA^1_\kk}$
  from 
  \eqref{eq:71} (for $n=1$), the morphism
  $\Phi'$ from 
  Theorem~\ref{t:total-vanishing-cycles-group-homo}.
  This makes up for
  Remark~\ref{rem:Phi'-cannot-be-ring-morphism}.
  As observed in Remarks~\ref{rem:Phi'-on-smooth-proper-morphisms}
  and~\ref{rem:Phi'-of-L-A1}, 
  $\Phi$ vanishes on many other elements, for example on
  $[\DA^1_\kk \xra{\id} \DA^1_\kk]$ or on
  $\DL_{\DA^1_\kk}$.
\end{remark}

\subsection{Compactification}
\label{sec:compactification}

For the convenience of the reader we recall our compactification
result from \cite{valery-olaf-matfak-motmeas}.

\begin{proposition}
  [{\cite[Prop.~6.1]{valery-olaf-matfak-motmeas}}]
  \label{p:compactification}
  Let $\kk$ be an algebraically closed field of characteristic zero.
  Let $X$ and $Y$ be smooth 
  $\kk$-varieties and let
  $V \colon  X \ra \DA^1_\kk$ and $W \colon Y\ra \DA^1_\kk$ be 
  projective
  morphisms (hence $X$ and $Y$ are
  quasi-projective $\kk$-varieties).
  Consider the convolution
  \begin{equation*}
    V \circledast W \colon X\times Y\xra{V \times W} \DA^1_\kk\times
    \DA^1_\kk\xra{+} \DA^1_\kk. 
  \end{equation*}
  Then there exists a smooth quasi-projective $\kk$-variety $Z$ with
  an open embedding $X\times Y \hra Z$ and a
  projective morphism $h \colon Z\ra \DA^1_\kk$ such that the
  following conditions are satisfied.
  \begin{enumerate}[label=(\roman*)]
  \item
    \label{enum:compactification-diagram}
    The diagram
    \begin{equation*}
      \xymatrix{
        {X\times Y} 
        \ar@{^{(}->}[r]
        \ar[d]^-{V \circledast W} & {Z} \ar[d]^-h\\
        {\DA^1_\kk} \ar@{}[r]|-{=} & {\DA^1_\kk}
      }
    \end{equation*}
    commutes.
  \item
    \label{enum:compactification-critical-points}
    All critical points of $h$ are contained
    in $X\times Y,$ i.\,e.\ $\Sing(V \circledast W)=\Sing(h)$.
  \item
    \label{enum:compactification-boundary-snc}
    We have $Z \setminus X \times Y =  \bigcup_{i=1}^s D_i$ where
    the
    $D_i$ are pairwise distinct smooth
    prime divisors.
    More precisely, $Z \setminus X \times Y$ is the support
    of a simple normal crossing divisor.
  \item
    \label{enum:compactification-intersections-proper-and-smooth}
    For every $p$-tuple $(i_1, \dots, i_p)$ of indices
    (with $p \geq 1$) the morphism
    \begin{equation*}
      h_{i_1 \dots i_p} \colon  D_{i_1 \dots i_p}:=D_{i_1} \cap \dots \cap
      D_{i_p} \ra \DA^1_\kk   
    \end{equation*}
    induced by $h$ is projective and smooth. In
    particular, all
    $D_{i_1 \dots i_p}$ are smooth quasi-projective $\kk$-varieties.
  \end{enumerate}
\end{proposition}

\section{Comparison with the matrix factorization motivic
  measure} 
\label{sec:comp-with-matr}

We would like to place
Theorem~\ref{t:total-motivic-vanishing-cycles-ring-homo} 
in a certain context and compare the motivic measure $\Phi$ 
or rather $\epsilon_! \circ \Phi$ with
another motivic measure of a different nature. 

\subsection{Categorical motivic measures}
\label{sec:categorical-measures}

First let us recall the {\it categorical} measure
\begin{equation*}
  \nu \colon K_0(\Var_\kk)\ra K_0(\sat^\DZ_\kk)  
\end{equation*}
constructed in \cite{bondal-larsen-lunts-grothendieck-ring}.
Here $K_0(\sat^\DZ_\kk)$ is the free abelian group generated by
quasi-equivalence classes of saturated differential
$\DZ$-graded 
$\kk$-categories with relations coming from semiorthogonal
decompositions into admissible subcategories on the level of
homotopy categories. The map $\nu $ sends the class $[X]$ of a
smooth projective $\kk$-variety $X$ to the class
$[\D^\bd(\Coh(X))]$ 
of (a suitable
differential $\DZ$-graded $\kk$-enhancement of) its derived
category
$\D^\bd(\Coh(X))$. The
tensor product of differential $\DZ$-graded $\kk$-categories induces a ring structure on
$K_0(\sat^\DZ_\kk)$ and $\nu $ is a ring homomorphism.  In recent
papers
\cite{valery-olaf-matfak-semi-orth-decomp,
  valery-olaf-matfak-motmeas} we have constructed a motivic
measure
\begin{equation*}
  \mu \colon 
  (K_0(\Var_{\DA^1_\kk}), \star)  
  \ra K_0(\sat_\kk^{{\DZ}_2})  
\end{equation*}
which is a relative analogue of the measure $\nu $. Here
$K_0(\sat_\kk^{{\DZ}_2})$ is defined in exactly the same way as
$K_0(\sat^\DZ_\kk)$ except that this time we consider saturated
differential ${{\DZ}_2}$-graded $\kk$-categories. If
$[X \xra{W} {\DA}^1_\kk]\in K_0(\Var_{{\DA}^1_\kk})$ where $X$ is
smooth over $\kk$ and $W$ is proper, then $\mu([X\xra{W}{\DA}^1_\kk])$
is defined as the class 
$[\bfMF(W)]$
of (a suitable differential $\DZ_2$-graded $\kk$-enhancement of)
the category
\begin{equation*}
  \bfMF(W):=\prod_{a\in k}\bfMF(X,W-a)^\natural  
\end{equation*}
Here $\bfMF(X,W-a)^\natural$ is the Karoubi envelope of the
category $\bfMF(X,W-a)$ of matrix factorizations of the potential
$W-a$.

The measures $\nu$ and $\mu$ are related by the commutative
diagram of ring homomorphisms 
(as announced in the introduction of \cite{valery-olaf-matfak-semi-orth-decomp})
\begin{equation}
  \label{eq:folding-diagram}
  \xymatrix{
    {K_0(\Var_\kk)} \ar[r]^-{\nu} \ar[d] 
    \ar[d]
    & {K_0(\sat_\kk^{\DZ})} \ar[d]\\
    {(K_0(\Var_{\DA^1_\kk}), \star)} \ar[r]^-{\mu} 
    & {K_0(\sat_\kk^{\DZ_2})}
  }
\end{equation}
where $K_0(\Var_\kk) \ra  K_0(\Var_{{\DA}^1_\kk})$ is the ring
homomorphism \eqref{eq:46} (for $n=1$)
and
$K_0(\sat^\DZ_\kk)  \ra  K_0(\sat_\kk^{{\DZ}_2})$ is induced from
the {\it folding} 
(see \cite{olaf-folding-derived-categories-in-prep})
which assigns to a differential ${\DZ}$-graded $\kk$-category the
corresponding differential ${{\DZ}}_2$-graded $\kk$-category. 

\subsection{Comparing vanishing cycles and matrix factorization
  measures}
\label{sec:comp-vanish-cycl}



To each saturated differential ${{\DZ}}_2$-graded $\kk$-category $A$
one assigns the finite dimensional ${{\DZ}}_2$-graded vector
space
\begin{equation*}
  \HP(A)=\HP_0(A)\oplus \HP_1(A)  
\end{equation*}
over the Laurent power series field $\kk((u))$ - the periodic
cyclic homology of $A$ (see \cite{keller-invariance}).

Put
$\chi_{\HP}(A):=\dim_{\kk((u))}\HP_0(A)-\dim_{\kk((u))}\HP_1(A)$. Since
$\HP$ is additive on semiorthogonal decompositions of
triangulated categories (see
\cite{keller-cyclic-homology-exact-cat}) 
this
assignment descends to a group homomorphism
\begin{equation*}
  \chi_{\HP}\colon K_0(\sat_{k}^{{\DZ}_2})\ra {\DZ}  
\end{equation*}
Because of the K\"unneth property for $\HP$ 
(see \cite{shklyarov-hodge-theoretic-property-kuenneth} and references therein) the map $\chi_{\HP}$ is in fact a ring homomorphism.

On the other hand, if $\kk={\DC}$ we have the usual ring
homomorphism (see \cite{looijenga-motivic-measures})
\begin{equation}
  \label{eq:chi_c}
  \chi_\opc:=
  \sum (-1)^i\dim \opH_\opc^i \ \colon \mathcal{M}_{\DC}\ra {\DZ}
\end{equation}
Notice that $\chi_\opc({\DL})=1$, hence $\chi_\opc$ is indeed well-defined as a homomorphism from
$\mathcal{M}_{\DC}$.

Forgetting the action of $\hat{\upmu}$ obviously defines a 
map
\begin{equation}
  \label{eq:14}
  \mathcal{M}_\DC^{\hat{\upmu}} \ra \mathcal{M}_\DC
\end{equation}
of $\mathcal{M}_\DC$-modules.
Clearly, this map is a ring
homomorphism if we equip its source with the usual
multiplication.
However, this is not true if we equip its source with the 
convolution product $*$ as we will explain in
Lemma~\ref{l:forgetful-convolution} below.
Nevertheless we have the following result.

\begin{proposition}
  \label{p:chi_c-covolution}
  The composition of $\chi_\opc$ (see \eqref{eq:chi_c}) with the map ``forget the
  \mbox{$\hat{\upmu}$-action}''
  \eqref{eq:14}
  defines a ring homomorphism
  \begin{equation}
    \label{eq:72}
    \chi_\opc \colon (\mathcal{M}_\DC^{\hat{\upmu}},*) \ra \DZ
  \end{equation}
  which we denote again by $\chi_\opc$.
\end{proposition}

\begin{proof}
  Let $A$ and $B$ be complex varieties
  with
  a good $\upmu_n$-action for some $n \geq 1$. We need to show
  that $A \times B$ and 
  \begin{equation*}
    [A]*[B] = 
    [((A \times^{\upmu_n} \underset{x}{\DGm})
    \times 
    (B \times^{\upmu_n} \underset{y}{\DGm}))|_{x^n+y^n=0}]
    -[((A \times^{\upmu_n} \underset{x}{\DGm})
    \times 
    (B \times^{\upmu_n} \underset{y}{\DGm}))|_{x^n+y^n=1}]
  \end{equation*}
  (see
  \eqref{eq:convolution-explicit})
  have the same Euler characteristic
  with compact support.
  Since $\DGm \ra \DGm$, $x \mapsto x^n$, is a $\upmu_n$-torsor
  (in the \'etale topology) we have a pullback square
  \begin{equation*}
    \xymatrix{
      {A \times \DGm}  
      \ar[rr]^-{(a,x) \mapsto x}
      \ar[d]
      &&
      {\DGm} 
      \ar[d]^-{x \mapsto x^n}
      \\
      {A \times^{\upmu_n} \DGm}  
      \ar[rr]^-{[a,x] \mapsto x^n}
      &&
      {\DGm.}
    }
  \end{equation*}
  Its lower horizontal morphism is an \'etale-locally
  trivial fibration with fiber $A$. Therefore it is a locally
  trivial fibration if we pass to the analytic topologies.
  In this way we obtain a locally trivial fibration
  \begin{equation*}
    f \colon (A^\an \times^{\upmu_n(\DC)} \DC^\times)
    \times 
    (B^\an \times^{\upmu_n(\DC)} \DC^\times)
    \xra{([a,x],[b,y]) \mapsto (x^n,y^n)} \DC^\times \times \DC^\times
  \end{equation*}
  with fiber $A^\an \times B^\an$. 
  Consider the subsets $N:=\{x'+y'=0\} \cong \DC^\times$ and 
  $E:=\{x'+y'=1\} \cong \DC^\times -\{1\}$ of the base
  of this
  fibration where
  $x'$ and $y'$ are the obvious coordinates.
  Then
  \begin{multline*}
    \chi_\opc([A]*[B]) =
    \chi_\opc(f^{-1}(N))-\chi_\opc(f^{-1}(E))\\
    =
    \chi_\opc(A^\an \times B^\an)
    (\chi_\opc(N) - \chi_\opc(E))
    =\chi_\opc(A^\an \times B^\an).
  \end{multline*}
  This proves what we need.
\end{proof}

Although not strictly needed for our purposes we would like to
include the following result (which is also true for $\kk$
instead of $\DC$).

\begin{lemma}
  \label{l:forgetful-convolution}
  The map ``forget the \mbox{$\hat{\upmu}$-action}''
  $f \colon \mathcal{M}_\DC^{\hat{\upmu}} \ra \mathcal{M}_\DC$
  (see
  \eqref{eq:14}) does not define a
  ring homomorphisms 
  $(\mathcal{M}_\DC^{\hat{\upmu}},*) \ra \mathcal{M}_\DC$.
  %
  %
\end{lemma}

\begin{proof}
  Let $M=\upmu_2 \in \Var^{\upmu_2}_\DC$ with obvious action of
  $\upmu_2$.
  We claim that $f([M] * [M]) \not= f([M])f([M])$.

  We clearly have $f([M])f([M])=4 [\Spec \DC]=4$.
  On the other hand multiplication defines an isomorphism
  $M \times^{\upmu_2} \DGm \sira \DGm$ and therefore
  \eqref{eq:convolution-explicit} yields
  \begin{equation*}
    [M] * [M] = 
    [(\underset{x}{\DGm} \times 
      \underset{y}{\DGm})|_{x^2+y^2=0}]
      -[(\underset{x}{\DGm} \times 
      \underset{y}{\DGm})|_{x^2+y^2=1}]
  \end{equation*}
  The $\upmu_2$-action on $\DGm \times \DGm$ is the diagonal action.
  Instead of using the coordinates $(x,y)$ on $\DA^2_\DC$ let us
  use the coordinates $(a,b)$ where $a=x+iy$ and $b=x-iy$.
  Then $x^2+y^2=ab$ and the conditions $x \not= 0$ and $y\not=0$
  are equivalent to $a+b \not=0$ and $a-b \not=0$. 
  Hence
  \begin{equation*}
    [M] * [M] 
    = 
    [\underset{(a,b)}{\DA^2_\DC}|_{ab=0,\; a \not=\pm b}]
    -[\underset{(a,b)}{\DA^2_\DC}|_{ab=1,\; a \not=\pm b}]
  \end{equation*}
  The $\upmu_2$-action on $\DA^2_\DC$ is again the diagonal
  action.
  The first summand is the coordinate cross without the origin
  and equal to $2[\DGm]$ with obvious $\upmu_2$-action. 
  To treat the second summand note that the map $(\DGm - \upmu_4)
  \ra  
  {\DA^2_\DC}|_{ab=1, a \not=\pm b}$, $a \mapsto (a, a^{-1})$
  defines a $\upmu_2$-equivariant isomorphism.
  Hence
  \begin{equation}
    \label{eq:mu2-convolution-squared}
    [M] * [M] = 2[\DGm]-[\DGm]+[\upmu_4]=[\DGm]+2[\upmu_2]
  \end{equation}
  and $f([M] * [M])=[\DGm]+4$.

  But the element $[\DGm]=f([M] * [M])-f([M])f([M])$ is certainly
  not zero in $\mathcal{M}_\DC$: taking the  
  Hodge-Deligne polynomial defines a ring homomorphism
  $\mathcal{M}_\DC \ra \DZ[u,v,u^{-1},v^{-1}]$ which sends
  $[\DGm]$ to $uv-1$, cf.\ \cite[Example
  4.6]{nicaise-sebag-grothendieck-ring-of-varieties}. 
\end{proof}

\begin{theorem} 
  \label{t:comparison}
  We have the following commutative diagram of ring homomorphisms
  \begin{equation*}
    \xymatrix{
      {(K_0(\Var_{\DA^1_\DC}), \star)} 
      \ar[r]^-{\mu} \ar[d]_-{\epsilon_!\circ \Phi}
      & {K_0(\sat_{\DC}^{{\DZ}_2})} \ar[d]^-{\chi_{\HP}} \\
      {(\mathcal{M}_{\DC}^{\hat{\upmu}},*)} \ar[r]^-{\chi_\opc} 
      & {\DZ}
    }
  \end{equation*}
  where the left vertical arrow is 
  the map 
  \eqref{eq:epsilon-Phi}
  from Theorem~\ref{t:total-motivic-vanishing-cycles-ring-homo}
  and the lower horizontal map is 
  the ring homomorphism \eqref{eq:72}.
\end{theorem}

\begin{proof} 
  The abelian group $K_0(\Var_{{\DA}^1_{\DC}})$ is
  generated by classes 
  $[X\xra{W} {\DA}^1_{\DC}]$ where $X$ is smooth over ${\DC}$ and
  the map $W$ is projective
  (see \cite{bittner-euler-characteristic}).
  So it suffices to prove commutativity on such generators.

  Fix a projective map $W \colon X\ra {\DA}^1_{\DC}$ of a smooth
  ${\DC}$-variety $X$. Then by definition
  \begin{equation*}
    \mu (W)=\sum_{a\in {\DC}}[\bfMF(X,W-a)^\natural ]
    \in
    K_0(\sat_{{\DC}}^{{\DZ}_2})   
  \end{equation*}
  and
  \begin{equation*}
    \epsilon_!\circ \Phi (W)
    =\sum_{a \in {\DC}}(\epsilon_a)_!\phi_{W,a}\in \mathcal{M}_{{\DC}}   
  \end{equation*}
  with notation as in \eqref{eq:phi-V-k}.
  So it suffices to prove that
  \begin{equation*} 
    \chi_{\HP}(\bfMF(X,W-a)^\natural)=\chi_\opc((\epsilon_a)_!\phi_{W,a})
  \end{equation*}
  for any given $a\in {\DC}$. We may and will assume that $a=0$.

  Let $X^{\an}$ denote the space $X$ with the analytic
  topology. Recall the classical functors of nearby and vanishing
  cycles
  \begin{equation*}
    \psi^{\geom}_W,\phi^{\geom}_W\colon \D^\bd_\opc(X^{\an})\to
    \D^\bd_\opc(X_0^{\an}) 
  \end{equation*}
  between the corresponding derived categories of constructible
  sheaves with complex coefficients.
  For $F\in \D^\bd_\opc(X^{\an})$ we have a distinguished triangle
  \begin{equation}
    \label{eq:extr}
    F\vert_{X_0^{\an}}\ra \psi^{\geom}_WF\ra \phi^{\geom}_WF\to
    F\vert_{X_0^{\an}}[1] 
  \end{equation}
  in  $\D^\bd_\opc(X_0^{\an})$ (see
  \cite[Exp.~XIII]{SGA7-II}).

  In particular for the constant sheaf ${\DC}_{X^{\an}}$ we have
  the complex $\phi^{\geom}_W{\DC}_{X^{\an}}$ of sheaves on
  $X_0^{\an}$. Consider its hypercohomology with compact
  supports
  $\opH_\opc^\bullet (X_0^{\an}, \phi^{\geom}_W{\DC}_{X^{\an}})$
  and its Euler characteristic
  $\sum_i(-1)^i\dim \opH^i_\opc(X_0^{\an}, \phi^{\geom}_W{\DC}_{X^{\an}})$.
  (Note that in our case we may as well consider the
  hypercohomology $\opH^\bullet$ instead of $\opH^\bullet_\opc$,
  since $X_0^\an$ is compact.) It follows from
  \cite[Thm.~1.1]{efimov-cyclic-homology-matrix-factorizations}
  that 
  \begin{equation*}
    \chi_{\HP}(\bfMF(X,W))
    =-\sum_i(-1)^i\dim \opH^i_\opc(X_0^{\an},\phi^{\geom}_W{\DC}_{X^{\an}}). 
  \end{equation*}
  By the localization theorem in cyclic homology it follows that the
  Karoubi closure $\bfMF(X,W)^\natural$ has the same cyclic
  homology as $\bfMF(X,W)$, i.\,e.\
  $\chi_{\HP}(\bfMF(X,W))=\chi_{\HP}(\bfMF(X,W)^\natural)$. 
  Hence it remains to prove the equality
  \begin{equation}
    \label{eq:val4} 
    \chi_\opc((\epsilon_0)_!\phi_{W,0})=-\sum_i(-1)^i\dim
    \opH^i_\opc(X_0^{\an}, \phi^{\geom}_W{\DC}_{X^{\an}}).
  \end{equation}

  \begin{lemma}
    \label{l:SH}
    \begin{enumerate}
    \item 
      \label{enum:SH-exist}
      For every variety $Y$ there
      exists a unique 
      group homomorphism
      \begin{equation*}
        \opSH_Y\colon K_0(\Var_Y)\ra K_0(\D_\opc^\bd(Y^{\an}))  
      \end{equation*}
      such that $\opSH_Y([Z\xra{f}Y])=[{\bR}f_!{\DC}_{Z^{\an}}]$.
    \item 
      \label{enum:SH-push}
      Given a morphism of varieties $g\colon Y\ra T$ the diagram
      \begin{equation*}
        \xymatrix{
          {K_0(\Var_Y)} \ar[r]^-{\opSH_Y} \ar[d]_-{g_!}
          & {K_0(\D_\opc^\bd(Y^{\an}))} \ar[d]^-{K_0({\bR}g_!)}\\
          {K_0(\Var_T)} \ar[r]^-{\opSH_T} 
          & {K_0(\D_\opc^\bd(T^{\an}))}
        }
      \end{equation*}
      commutes.
    \item
      \label{enum:SH-C}
      If $Y=\Spec {\DC}$, then $K_0(\D^\bd_\opc((\Spec
      {\DC})^\an))={\DZ}$ (by taking the alternating sum of the
      cohomologies) and $\opSH_{\Spec {\DC}}([Z\ra \Spec
      {\DC}])=\chi_\opc([Z])$. 
    \end{enumerate}
  \end{lemma}

  \begin{proof} 
    \ref{enum:SH-exist}
    For a variety $S$ and an open embedding $j\colon
    U\hookrightarrow S$ with 
    complementary closed embedding $i\colon
    Z=S-U\hookrightarrow S$ recall the short exact sequence 
    of sheaves
    \begin{equation*}
      0\ra j_!{\DC}_{U^{\an}}\ra {\DC}_{S^{\an}}\ra i_!{\DC}_{Z^{\an}}\ra 0.
    \end{equation*}
    This implies that the map
    $\opSH_Y([Z\xra{f}Y])=[{\bR}f_!{\DC}_{Z^{\an}}]$ indeed
    descends to a homomorphism $\opSH_Y\colon K_0(\Var_Y)\to
    K_0(\D_\opc^\bd(Y^{\an}))$. Uniqueness is obvious. 

    \ref{enum:SH-push}
    Given a morphism $f\colon Z\ra Y$ we have by definition
    \begin{equation*}
      K_0({\bR}g_!)\cdot \opSH_Y([Z\xra{f}Y])=
      K_0({\bR}g_!)[{\bR}f_!{\DC}_{Z^{\an}}]=[{\bR}(gf)_!{\DC}_{Z^{\an}}]  
    \end{equation*}
    and
    \begin{equation*}
      \opSH_T \cdot g_!([Z\xra{f}Y])
      =\opSH_T([Z\xra{gf}T]=[{\bR}(gf)_!{\DC}_{Z^{\an}}]  
    \end{equation*}

    \ref{enum:SH-C}
    This is clear.
  \end{proof}
  
  Now \cite[Prop.~3.17]{guibert-loeser-merle-convolution} implies
  the following equality in 
  $K_0(\D^\bd_\opc(X^{\an}_0))$: 
  \begin{equation*}
    \opSH_{X_0}(\psi_{W,0})=[\psi_W^{\geom}({\DC}_{X^{\an}})].
  \end{equation*}
  Applying part \ref{enum:SH-push}
  of Lemma~\ref{l:SH}
  to the map $\epsilon_0\colon
  X_0\ra {\Spec {\DC}}$ and
  using part \ref{enum:SH-C}
  we conclude that
  \begin{equation*}
    \chi_\opc((\epsilon_0)_!\psi_{W,0})=\sum_i(-1)^i\dim
    \opH^i_\opc(X_0^{\an},\psi_W^{\geom}({\DC}_{X^{\an}})). 
  \end{equation*}
  Notice that on one hand by definition of $\phi_{W,0}$ we have
  \begin{equation*}
    \chi_\opc((\epsilon_0)_!\phi_{W,0})
    =\chi_\opc((\epsilon_0)_![X_0\xra{\id}
    X_0])-\chi_\opc((\epsilon_0)_!\psi_{W,0})    
  \end{equation*}
  and on the other hand by the distinguished triangle
  \eqref{eq:extr} we have 
  \begin{multline*}
    \sum_i(-1)^i\dim \opH^i_\opc(X_0^{\an},\phi_W^{\geom}({\DC}_{X^{\an}}))=
    \sum_i(-1)^i\dim
    \opH^i_\opc(X_0^{\an},\psi_W^{\geom}({\DC}_{X^{\an}}))
    \\
    -\sum_i(-1)^i\dim \opH^i_\opc(X_0^{\an},{\DC}_{X_0^{\an}})  
  \end{multline*}
  It remains to notice that
  \begin{equation*}
    \chi_\opc((\epsilon_0)_![X_0\xra{\id} X_0])=\sum_i(-1)^i\dim
    \opH^i_\opc(X_0^{\an},{\DC}_{X_0^{\an}})   
  \end{equation*}
  This proves equality \eqref{eq:val4} and finishes the proof of the
  theorem. 
\end{proof}

We give two simple examples in which the equality \eqref{eq:val4} can
be verified directly. 

\begin{example} 
  \label{expl:a^n}
  Let $X={\DA}^1_\DC$ and $W(a)=a^n$ for some $n\geq 1$.
  Then $\phi^{\geom}_W{\DC}_{X^{\an}}={\DC}^{\oplus n-1}_{(0)}$. 
  Hence the right-hand side of 
  equation \eqref{eq:val4} is equal to $-(n-1)$.

  On the other hand, in the notation of 
  the proof of
  Proposition~\ref{p:motivic-vanishing-cycles-over-singular-part}
  (with the identity as embedded resolution) 
  the divisor $E$ is $n\cdot (0)$ and
  hence its $\upmu_n$-Galois covering $\tilde{E}$ is
  isomorphic to $\upmu_n$. 
  From \eqref{eq:formula-mot-nearby-fiber}
  we obtain
  \begin{equation*}
    \phi_{W,0}=[|X_0|]-\psi_{W,0}=[(0)]-\upmu_n.  
  \end{equation*}
  Thus $\chi_\opc((\epsilon_0)_!\phi_{W,0})$ is also equal to $-(n-1)$.
\end{example}

\begin{example} 
  \label{expl:ab}
  Let $X={\DA}^2_\DC$ and $W\colon X\ra {\DA}^1_\DC$,
  $W(a,b)=ab$. (This is not proper, but should make no 
  difference since the complex $\phi^{\geom}_W{\DC}_{X^{\an}}$
  has compact support.) 
  Then $\phi^{\geom}_W{\DC}_{X^{\an}}={\DC}_{(0,0)}[-1]$. Hence
  the right-hand side of \eqref{eq:val4} is equal to 1. 

  On the other hand, in the notation of 
  the proof of
  Proposition~\ref{p:motivic-vanishing-cycles-over-singular-part}
  the divisor $E$ is
  the coordinate cross (with components of multiplicity one) 
  and so  \eqref{eq:formula-mot-nearby-fiber}
  yields
  \begin{equation*}
    \phi_{W,0}=[X_0]-\psi_{W,0}=(\DGm +\DGm+pt)-(\DGm+\DGm-\DGm)={\DL}  
  \end{equation*}
  Hence $\chi_\opc((\epsilon_0)_!\phi_{W,0})=1$.

  Here is another way to compute this example.
  Using coordinates $(s,t)$ on $\DA^2_\DC$ so that $a=s+it$ and
  $b=s-it$ we have $W(a,b)=ab=s^2+t^2= s^2 \circledast t^2$. 
  Example~\ref{expl:a^n} shows that
  $\phi_{s^2,0}=[(0)]-\upmu_2$ and
  $\chi_\opc((\epsilon_0)_!\phi_{s^2,0})=
  -1$.
  We have $\epsilon_!\Phi(s^2)=(\epsilon_0)_!\phi_{s^2,0}$
  and
  \begin{equation*}
    \chi_\opc((\epsilon_0)_!\phi_{W,0})=
    \chi_\opc(\epsilon_!\Phi(ab))=
    \chi_\opc(\epsilon_!\Phi(s^2))
    \chi_\opc(\epsilon_!\Phi(t^2))
    =(-1)^2=1
  \end{equation*}
  using multiplicativity of our motivic measures.
  We can also use the motivic Thom-Sebastiani 
  Theorem~\ref{t:thom-sebastiani-GLM-motivic-vanishing-cycles}
  and compute 
  (use Remark~\ref{rem:action-on-B-trivial}
  and (the computation leading to) equation 
  \eqref{eq:mu2-convolution-squared})
  \begin{multline*}
    \phi_{W,0}=
    \Psi(\phi_{s^2,0} \times \phi_{t^2,0})=
    \Psi([(0)] \times [(0)])
    -
    2 \Psi([(0)] \times [\upmu_2])
    + \Psi([\upmu_2] \times [\upmu_2])\\
    = [(0)]-2 [\upmu_2]+([\DGm] +2[\upmu_2])) =\DL.
  \end{multline*}
  Here the $\mu_2$-action on $\DG_m$ is a priori the obvious one
  but can then also be assumed to be trivial by the defining
  relations 
  of the equivariant Grothendieck group.
\end{example}

\subsection{Summarizing diagram}
\label{sec:summarizing-diagram}

We collect the motivic measures considered in this paper in the following
commutative diagram (in case $\kk={\DC}$; see
\eqref{eq:folding-diagram} and Theorem~\ref{t:comparison}).
\begin{equation*}
  \xymatrix{
    {K_0(\Var_\DC)} \ar[r]^-{\nu} \ar[d]
    & {K_0(\sat_\DC^{\DZ})} \ar[d]\\
    {(K_0(\Var_{\DA^1_\DC}), \star)} \ar[r]^-{\mu} \ar[d]_{\epsilon_!
      \circ \Phi}
    & {K_0(\sat_\DC^{\DZ_2})} \ar[d]^-{\chi_\HP}
    \\
    {(\mathcal{M}^{\hat{\upmu}}_\DC, *)} 
    \ar[r]^-{\chi_\opc}
    & {\DZ}
  }
\end{equation*}
The upper left vertical arrow and the vertical composition on the
left are the algebra structure maps. The composition from the top
left corner to the bottom right corner is induced by mapping a
complex variety to its Euler characteristic with compact support.


\def\cprime{$'$} \def\cprime{$'$} \def\cprime{$'$} \def\cprime{$'$}
  \def\Dbar{\leavevmode\lower.6ex\hbox to 0pt{\hskip-.23ex \accent"16\hss}D}
  \def\cprime{$'$} \def\cprime{$'$}
\providecommand{\bysame}{\leavevmode\hbox to3em{\hrulefill}\thinspace}
\providecommand{\MR}{\relax\ifhmode\unskip\space\fi MR }
\providecommand{\MRhref}[2]{%
  \href{http://www.ams.org/mathscinet-getitem?mr=#1}{#2}
}
\providecommand{\href}[2]{#2}

\end{document}